\documentclass[11pt,francais]{article}
\usepackage{amsmath}\usepackage{epsf,amsfonts,amsthm}\usepackage{amscd,mathptmx,amssymb}
\usepackage{xcolor,epic,eepic}\usepackage{epsfig}
\usepackage{fontenc,indentfirst, delarray,amsfonts,amsmath,amssymb}
\usepackage{rotating}
\usepackage{mathdots}
\usepackage[T1]{fontenc}
\usepackage[matrix,arrow,curve]{xy}
\usepackage{amsmath}
\usepackage{amssymb}
\usepackage{amsthm}
\usepackage{amscd}
\usepackage{amsfonts}
\usepackage{graphicx}%
\usepackage{fancyhdr}
\usepackage{dsfont,texdraw}
\usepackage{amsmath}\usepackage{epsf,amsfonts,amsthm}\usepackage{amscd,mathptmx,amssymb}
\usepackage{xcolor,epic,eepic}\usepackage{epsfig}
\usepackage{fontenc,indentfirst, delarray,amsfonts,amsmath,amssymb}
\usepackage{rotating}
\usepackage[T1]{fontenc}
\newcommand{\bee}{\begin{enumerate}}
\newcommand{\eee}{\end{enumerate}}
\newcommand{\benn}{\begin{equation*}}
\newcommand{\eenn}{\end{equation*}}
\newcommand{\be}{\begin{equation}}
\newcommand{\ee}{\end{equation}}
\newcommand{\bean}{\begin{eqnarray}}
\newcommand{\eean}{\end{eqnarray}}
\newcommand{\bea}{\begin{eqnarray*}}
\newcommand{\eea}{\end{eqnarray*}}
\newcommand{\w}{\wedge}

\newcommand{\p}{\partial}
\newcommand{\la}{\langle}
\newcommand{\ra}{\rangle}
\newcommand{\raa}{\rightarrow}

\newcommand{\Ci}{C^{\infty}}

\newcommand{\N}{\mathbb{N}}
\newcommand{\Z}{\mathbb{Z}}
\newcommand{\R}{\mathbb{R}}

\newcommand{\Q}{\mathbb{Q}}

\newcommand{\lp}{\left(}
\newcommand{\rp}{\right)}

\newcommand{\op}[1]{\!\!\mathop{\rm ~#1}\nolimits}

\mathchardef\za="710B  
\mathchardef\zb="710C  
\mathchardef\zg="710D  
\mathchardef\zd="710E  
\mathchardef\zve="710F 
\mathchardef\zz="7110  
\mathchardef\zh="7111  
\mathchardef\zy="7112 

\mathchardef\zi="7113  
\mathchardef\zk="7114  
\mathchardef\zl="7115  
\mathchardef\zm="7116  
\mathchardef\zn="7117  
\mathchardef\zx="7118  
\mathchardef\zp="7119  
\mathchardef\zr="711A  
\mathchardef\zs="711B  
\mathchardef\zt="711C  
\mathchardef\zu="711D  
\mathchardef\zf="711E 
\mathchardef\zq="711F  
\mathchardef\zc="7120  
\mathchardef\zw="7121  
\mathchardef\ze="7122  
\mathchardef\zvy="7123  
\mathchardef\zvw="7124  
\mathchardef\zvr="7125 
\mathchardef\zvs="7126 
\mathchardef\zvf="7127  
\mathchardef\zG="7000  
\mathchardef\zD="7001  
\mathchardef\zY="7002  
\mathchardef\zL="7003  
\mathchardef\zX="7004  
\mathchardef\zP="7005  
\mathchardef\zS="7006  
\mathchardef\zU="7007  
\mathchardef\zF="7008  
\mathchardef\zW="700A  

\newcommand{\cyclic}{\mathop{\kern0.9ex{{+}
\kern-2.15ex\raise-.25ex\hbox{\Large\hbox{$\circlearrowright$}}}}\limits}

 \newcommand{\cA}{{\cal A}}
 \newcommand{\cM}{{\cal M}}

\newtheorem{theo}{Theorem}
\newtheorem{prop}{Proposition}
\newtheorem{lem}{Lemma}
\newtheorem{cor}{Corollary}

\newtheorem{defi}{Definition}

\begin{document}
\title{{\bf On the category of Lie $n$-algebroids}}
\date{}
\author{Giuseppe Bonavolont\`a and Norbert Poncin\thanks{G. Bonavolont\`a thanks the Luxembourgian National Research Fund for support via AFR grant 09-072. The research of N. Poncin was supported by Grant
GeoAlgPhys 2011-2013 awarded by the University of Luxembourg. Moreover, N. Poncin thanks the Institute of Mathematics of the Polish Academy in Warsaw and especially J. Grabowski for an invitation to a research stay during which part of this work was done.}}

\maketitle

\newcommand{\fg}{{\frak g}}
\newcommand{\fh}{{\frak h}}
\newcommand{\mk}{{\mathds{k}}}
\newcommand{\0}{\otimes}

\begin{abstract} Lie $n$-algebroids and Lie infinity algebroids are usually thought of exclusively in supergeometric or algebraic terms. In this work, we apply the higher derived brackets construction to obtain a geometric description of Lie $n$-algebroids by means of brackets and anchors. Moreover, we provide a geometric description of morphisms of Lie $n$-algebroids over different bases, give an explicit formula for the Chevalley-Eilenberg differential of a Lie $n$-algebroid, compare the categories of Lie $n$-algebroids and NQ-manifolds, and prove some conjectures of Sheng and Zhu \cite{SZ11}.\end{abstract}\vspace{5mm}

{\small 

\noindent{\bf Keywords :} Lie $n$-algebroids, split NQ-manifolds, morphisms, Chevalley-Eilenberg complex, graded symmetric tensor coalgebra, higher derived brackets, Lie infinity (anti)-algebra..}

\tableofcontents

\section{Introduction}

\subsection{General background}

The starting point of this work is the paper \cite{BKS} by Bojowald, Kotov, and Strobl. The authors prove that for the Poisson sigma model ({\small PSM}) a bundle map is a solution of the field equations if and only if it is a map of Lie algebroids, i.e. a morphism of Q-manifolds, or, as well, of differential graded algebras ({\small DGA}). Moreover, gauge equivalent solutions are homotopic maps of Lie algebroids or homotopic morphisms of Q-manifolds.\medskip

In case of the AKSZ sigma model, the target space is a symplectic Lie $n$-algebroid \cite{AKSZ97}, see also \cite{Sev01}. The concept of Lie $n$-algebroid or, more generally, of Lie infinity algebroid can be discussed in the cohesive $(\infty,1)$-topos of synthetic differential $\infty$-groupoids. Presentations by {\small DGA}-s and simplicial presheaves can be given. However, in contrast with Lie algebroids, no interpretation in terms of brackets seems to exist -- the approach is essentially algebraic.\medskip

Lie infinity algebroids are not only homotopifications of Lie algebroids, but also horizontal categorifications of Lie infinity algebras. Truncated Lie infinity algebras are themselves tightly connected with vertical categorifications of Lie algebras \cite{BC04}, \cite{KMP11}. Beyond this categorical approach to Lie infinity algebras, there exists a well-known operadic definition that has the advantage to be valid also for other types of algebras: if $P$ denotes a quadratic Koszul operad, the $P_{\infty}$-operad is defined as the cobar construction $\zW P^{\text{\;!`}}$ of the Koszul-dual cooperad $P^{\;\text{!`}}$ of $P$. A $P_{\infty}$-structure on a graded vector space $V$ is then a representation on $V$ of the differential graded operad $P_{\infty}$. On the other hand, a celebrated result by Ginzburg and Kapranov states that $P_{\infty}$-structures on $V$ are 1:1 (in the finite-dimensional context) with differentials $$d\in
\op{Der}^1({P^{!}}(sV^*))\quad \text{ or }\quad d\in \op{Der}^{-1}({P^{!}}(s^{-1}V^*))\;,$$or, also, 1:1 with codifferentials
$$D\in\op{CoDer}^{1}(P^{\;\text{!`}}(s^{-2}V))\quad \text{
or }\quad D\in\op{CoDer}^{-1}({P^{\;\text{!`}}}(V))\;.$$Here $\op{Der}^1({P^{!}}(sV^*))$ (resp., $\op{CoDer}^{1}({P^{\;\text{!`}}}(s^{-2}V))$), for instance, denotes the space of endomorphisms of the free graded algebra over the Koszul dual operad $P^{!}$ of $P$ on the suspended linear dual $sV^*$ of $V$, which have degree 1 and are derivations with respect to each binary operation in $P^{!}$ (resp., the space of endomorphisms of the free graded coalgebra on the desuspended space $s^{-1}V$ that are coderivations) (by differential and codifferential we mean of course a derivation or coderivation that squares to 0). In case $P$ is the Lie operad, this implies that we have a 1:1 correspondence between Lie infinity structures on $V$ and degree 1 differentials $D\in \op{Der}^1(\odot (sV^*))$ of the free graded symmetric algebra over $sV^*$ or degree 1 codifferentials $D\in \op{CoDer}^1(\odot^c (s^{-1}V))$ of the free graded symmetric coalgeba over $s^{-1}V$. The latter description can also be formulated in terms of formal supergeometry (which goes back to Kontsevich's work on Deformation
Quantization of Poisson Manifolds): a Lie infinity structure on $V$ is the same concept as a homological vector field on the formal supermanifold $s^{-1}V$.\medskip

Each one of the preceding approaches to Lie infinity algebras has its advantages: a number of notions are God-given in the categorical setting, operadic techniques favor conceptual and `universal' ideas, morphisms tend to live in the algebraic or supergeometric world... However: Lie infinity algebras were originally defined by Lada and Stasheff \cite{LS93} in terms of infinitely many brackets. The same holds true for Lie algebroids: although they are exactly Q-manifolds, i.e. homological vector fields $Q\in \op{Der}^1(\zG(\w E^*))$ of a split supermanifold $E[1]$, where $E$ is a vector bundle, their original nature is geometric: they are vector bundles with an anchor map and a Lie bracket on sections that verifies the Leibniz rule with respect to the module structure of the space of sections. Moreover, in case of the {\small PSM}, Physics is formulated in this geometric setting.

\subsection{Main results and structure}

In this work, we describe Lie infinity algebroids and their morphisms in terms of anchors and brackets, thus providing a geometric meaning of concepts usually dealt with exclusively in the algebraic or supergeometric frameworks. \medskip

Section 2 begins with some information on graded symmetric and graded antisymmetric tensor algebras of $\Z$-graded modules over $\Z$-graded rings. Further, we prove that any N-manifold is noncanonically split, {\bf Theorem \ref{BatchelorNMan}}, and give examples of canonically split N-manifolds.\medskip

In Section 3, we study the standard and the homological gradings of the structure sheaf and the sheaf of vector fields of split N-manifolds, as well as the corresponding Euler fields. We also review graded symmetric tensor coalgebras, their coderivations and cohomomorphisms. We apply the higher derived brackets method to obtain from an NQ-manifold the geometric anchor-bracket description of a Lie $n$-algebroid, see {\bf Definition \ref{DefLieNAld}}. We thus recover and justify a definition given in \cite{SZ11}. There is a 1:1 correspondence between `geometric' Lie $n$-algebroids and NQ-manifolds, see {\bf Theorem \ref{MTheo2}}. The proof leads to the explicit form of the Chevalley-Eilenberg differential of a Lie $n$-algebroid, see {\bf Definition \ref{LieNAldCE}}. It reduces, for $n=1$, to the Lie algebroid de Rham cohomology operator, see {\bf Remark 8}. Moreover, it is shown that any Lie $n$-algebroid is induced by derived brackets. Some of these results may be considered as natural. However, their proofs are highly complex. Even the concept of `geometric' Lie $n$-algebroid was not known before 2011.\medskip

The last section contains the definitions of morphisms of `geometric' Lie $n$-algebroids over different and over isomorphic bases, see {\bf Definitions \ref{LadMorDef}, \ref{BasePresLieMorphDef}} and {\bf Remarks 9, 10}. For $n=1$, they reduce to Lie algebroid morphisms \cite{Mac05} and, over a point, we recover Lie infinity algebra morphisms \cite{AP10}. To justify these definitions, we prove that split Lie $n$-algebroid morphisms are exactly morphisms of NQ-manifolds between split NQ-manifolds, see {\bf Theorem \ref{LadMorTheo}}.

\section{N-manifolds}

\subsection{$\Z$-Graded symmetric tensor algebras}

The goal of this subsection is to increase readability of our text. The informed reader may skip it. For the $\Z_2$-graded case, see \cite{Man88}.\medskip

In the following, we denote by $\vee$ (resp., $\w$, $\odot$, $\boxdot$) the symmetric (resp., antisymmetric, graded symmetric, graded antisymmetric) tensor product. More precisely, let $M=\oplus_iM_i$ be a $\Z$-graded module over a $\Z$-graded commutative unital ring $R$. Graded symmetric (resp., graded antisymmetric) tensors on $M$ are defined, exactly as in the nongraded case, as the quotient of the tensor algebra $TM=\oplus_pM^{\0 p}$ by the ideal $$I=(m\0 n-(-1)^{mn}n\0 m)\quad (\text{resp., } I=(m\0 n+(-1)^{mn}n\0 m))\;,$$where $m,n\in M$ are homogeneous of degree denoted by $m,n$ as well. The tensor module $TM$ admits the following decomposition: $$TM=\bigoplus_{p\in\N}\bigoplus_{i_1\le\ldots\le i_p}M_{i_1\ldots i_p}:=\bigoplus_{p\in\N}\bigoplus_{i_1\le\ldots\le i_p}\lp \bigoplus_{\zs\in \op{Perm}} M_{\zs_{i_1}}\0\ldots\0 M_{\zs_{i_p}}\rp\;,$$where Perm denotes the set of all permutations (with repetitions) of the elements $i_1\le\ldots \le i_p.$ For instance, if $M=M_0\oplus M_1\oplus M_2$, the module $M^{\0 3}$ is given by $$M^{\0 3}=M_0\0 M_0\0 M_0\oplus \lp M_0\0 M_0\0 M_1\oplus M_0\0 M_1\0 M_0 \oplus M_1\0 M_0\0 M_0\rp\oplus\ldots$$The ideal $I$ is homogeneous, i.e. $$I=\bigoplus_{p\in\N}\bigoplus_{i_1\le\ldots\le i_p}\lp M_{i_1\ldots i_p} \cap I \rp\;,$$which actually means that the components of the decomposition of any element of $I\subset TM$ are elements of $I$ as well. To check homogeneity, it suffices to note that an element of $I$, e.g. $$(m_0\0 m_1-m_1\0 m_0)\0 (m'_0 + m'_1)\;,$$where the subscripts denote the degrees, has components $$(m_0\0 m_1\0 m'_0-m_1\0 m_0\0 m'_0) + (m_0\0 m_1\0 m'_1-m_1\0 m_0\0 m'_1)\;$$ in $I$, since we accept permutations inside the components. It follows that the graded symmetric tensor algebra $\odot M=TM/I$ reads $$\odot M=\bigoplus_{p\in\N}\bigoplus_{i_1\le\ldots\le i_p} M_{i_1}\odot\ldots\odot M_{i_p}, \quad\text{where}\quad  M_{i_1}\odot\ldots\odot M_{i_p}:=M_{i_1\ldots i_p} / (M_{i_1\ldots i_p}\cap I)\;.$$The `same' result holds true in the graded antisymmetric situation.

Note that, if $R$ is a $\Q$-algebra, the module $M_0\odot M_1$ is isomorphic to $M_0\otimes M_1$, via $$[{(\tiny 1/2)}(m_0\0 m_1 + m_1\0 m_0)]\mapsto m_0\0 m_1\;.$$ As the tensors $m_0\0 m_1$ and ${\tiny \frac{1}{2}}(m_0\0 m_1 + m_1\0 m_0)$ differ by ${\tiny \frac{1}{2}}(m_0\0 m_1 - m_1\0 m_0)$ and thus coincide in the quotient $M_0\odot M_1$, the identification $M_0\odot M_1\simeq M_0\0 M_1$ corresponds to the choice of a specific representative. Moreover, we have module isomorphisms of the type $$M_0\odot M_0\odot M_0\simeq \vee^3M_0\;,\quad M_0\odot M_0\odot M_1\simeq \vee^2M_0\otimes M_1\;,\quad M_0\odot M_1\odot M_1\simeq M_0\otimes \w^2M_1\;.$$Similarly, $$M_0\boxdot M_0\boxdot M_0\simeq \w^3M_0\;,\quad M_0\boxdot M_0\boxdot M_1\simeq \w^2M_0\otimes M_1\;,\quad M_0\boxdot M_1\boxdot M_1\simeq M_0\otimes \vee^2M_1\;.$$

\subsection{Batchelor's theorem for N-manifolds}

The results of this section are well-known. We consider their proofs as a warm-up exercise (although we did not find them in the literature).

\begin{defi} An \emph{ N-manifold} or \emph{$\N$-graded manifold} of degree $n$ (and
dimension $p|q_1|\ldots|q_n$) is a smooth Hausdorff
second-countable (hence paracompact) manifold $M$ endowed
with a sheaf ${\cal A}$ of $\N$-graded commutative associative
unital $\R$-algebras, whose degree 0 term is
$\cA^0=\Ci_{M}$ and which is locally freely generated, over
the corresponding sections of $\Ci_M$ (i.e. over functions of $p$ variables $x^j$ of degree 0), by
$q_1,\ldots, q_n$ graded commutative generators
$\xi_1^{j_1},\ldots,\xi_n^{j_n}$ of degree
$1,\ldots,n$, respectively.\end{defi}

If necessary, we assume -- for simplicity -- that $M$ is connected. Moreover, it is clear that coordinate transformations are required to preserve
the $\N$-degree. An alternative definition of N-manifolds by coordinate charts and
transition maps can be given.\medskip


\noindent{\bf Example 1.} Just as a vector bundle with shifted parity in the fibers is a (split) supermanifold, a graded vector bundle with shifted degrees is a (split) N-manifold. More precisely, let $E_{-1},E_{-2},\ldots,E_{-n}$ be smooth vector bundles of finite rank $q_1,\ldots,q_n$ over a same smooth (Hausdorff, second countable) manifold $M$. If we assign the degree $i$ to the fiber coordinates of $E_{-i}$, we get an N-manifold $E_{-i}[i]$. The direct sum $E=\oplus_{i=1}^n E_{-i}$ (resp., $E[\cdot]=\oplus_{i=1}^n E_{-i}[i]$) is a graded
vector bundle concentrated in degrees $-1$ to $-n$, with local coordinates of degree 0 (resp., with local fiber coordinates of degrees $1,\ldots, n$. We denote the graded symmetric tensor algebra of the dual bundle $E^*=\oplus_{i=1}^n
E^*_{-i}$, where the vectors of $E_{-i}^*$ have degree
$i$, by $\odot E^*.$ Consider first the case $n=2$. We have
$$\odot E^*=\odot (E^*_{-1}\oplus E_{-2}^*)=\odot E_{-1}^*\otimes \odot
E_{-2}^*=\w E_{-1}^*\0 \vee E_{-2}^*\;.$$Hence, $$\cA:=\zG(\odot E^*)=\zG(\w
E^*_{-1}\0 \vee E_{-2}^*)=\oplus_{n\in \N}\oplus_{k+2\ell=n}\,\zG(
\w^{k} E^*_{-1}\0 \vee^{\ell} E_{-2}^*)\;.$$ In particular,
$$\cA^0=\Ci_M,\; \cA^1=\zG(E_{-1}^*),\; \cA^2=\zG(\w^2 E_{-1}^*\oplus
E_{-2}^*),\;\ldots$$Note that, if we work over a common local
trivialization domain $U\subset M$ of $E_{-1}$ and $E_{-2}$, and
denote the base coordinates by $x$ and the fiber coordinates of $E_{-1}$ and $E_{-2}$ by
$\xi$ and $\zh$, respectively ($\xi$ and $\zh$ are then also the base vectors of the fibers of $E_{-1}^*$ and $E_{-2}^*$), a function $f$ in $\cA^2(U)$ reads
$$f(x,\xi,\zh)=\sum_{i<j}f_{ij}(x)\xi^i\xi^j+\sum_af_a(x)\zh^a\;.$$Similar results hold of course for any $n$. Eventually, the sheaf $(M,\cA)$ is an N-manifold of degree $n$. It is clear that this manifold should be denoted by $E[\cdot]=\oplus_{i=1}^n E_{-i}[i]$.

\begin{defi} A N-manifold $E[\cdot]=\oplus_{i=1}^n E_{-i}[i]$ induced by a graded vector bundle is called a \emph{split N-manifold}.\end{defi}

Batchelor's theorem \cite{Bat80} states that in the smooth category any supermanifold is noncanonically diffeomorphic to a split supermanifold. A similar result is known to hold true for N-manifolds.

\begin{theo}\label{BatchelorNMan} Any N-manifold $(M,\cA)$ of degree $n$ is noncanonically diffeomorphic to a split N-manifold $E[\cdot]$, where $E=\oplus_{i=1}^nE_{-i}$ is a graded vector bundle over $M$ concentrated in degrees $-1,\ldots,-n$.\end{theo}

\begin{proof}[Sketch of Proof]
Consider first an N-manifold $(M,\cA)$ of degree $n=2$. Since ${\cal A}^0=\Ci_M$ and ${\cal A}^0{\cal A}^1\subset {\cal A}^1$, the sheaf ${\cal A}^1$ is a locally free sheaf of
$\Ci_M$-modules and 
$${\cal A}^1\simeq \zG(E_{-1}^*),$$for some vector bundle $E_{-1}\to M$. Let
now ${\cal A}_1$ be the subalgebra of $\cA$ generated by ${\cal
A}^0\oplus{\cal A}^1$. Clearly, \be\label{NMfd1}\cA_1={\cal A}^0\oplus{\cal
A}^1\oplus (\cA^1)^2\oplus\ldots\simeq \zG(\w E_{-1}^*)\ee and ${\cal
A}_{1}\cap {\cal A}^2=(\cA^1)^2$ is a proper $\cA^0$-submodule of
$\cA^2$. Since the quotient $\cA^2/(\cA^1)^2$ is a locally free
sheaf of $\Ci_M$-modules (since locally $\cA^2$ is a direct sum), we have $${\cal A}^2/({\cal A}^1)^2\simeq
\zG(E_{-2}^*),$$where $E_{-2}$ is a vector bundle over $M$. The short
exact sequence \be 0\raa ({\cal A}^1)^2\raa {\cal A}^2\raa
\zG(E_{-2}^*)\raa 0\label{Ses}\ee of ${\cal A}^0$-modules is non
canonically split. Indeed, the Serre-Swan theorem states that, if $N$ is a smooth (Hausdorff, second-countable) manifold, a
$\Ci(N)$-module is finitely generated and projective if and only
if it is the module of smooth sections of a smooth vector bundle
over $N$, see e.g. \cite{Nes}[Theo.11.32], \cite{GMS}. The aforementioned
splitting of the sequence (\ref{Ses}) is a direct consequence
of the fact that $\zG(E_{-2}^*)$ is projective. Let us {\em fix a
splitting} and identify $\zG(E_{-2}^*)$ with a submodule of ${\cal
A}^2$:
$${\cal A}^2 = ({\cal A}^1)^2\oplus\zG(E_{-2}^*)= \zG(\w^2E^*_{-1}\oplus
E^*_{-2})\;.$$It follows that the subalgebra $\cA_2$, which is
generated by $\cA^0\oplus \cA^1\oplus \cA^2$, reads
$$\cA_2=\zG(\w E_{-1}^*\otimes \vee E_{-2}^*)\;.$$In the case $n=2$, the algebra $\cA_2$ is of
course the whole function algebra $\cA$. For higher $n$, we
iterate the preceding approach and obtain finally, modulo choices
of splittings, that {$${{\cal A}=\zG(\w
E^*_{-1}\otimes{\vee}E^*_{-2}\otimes\w E^*_{-3}\otimes\ldots)}=\zG(\odot E^*)\;.$$}Hence, the
considered N-manifold is diffeomorphic, modulo the
chosen splittings, to the split N-manifold $E[\cdot]=\oplus_{i=1}^nE_{-i}[i]\,.$
\end{proof}

\noindent{\bf Remarks 1.} \begin{itemize}\item The preceding theorem means that N-manifolds of degree $n$ together with a choice of splittings are
1:1 with graded vector bundles concentrated in degrees $-1,\ldots, -n$.

\item In view of Equation (\ref{NMfd1}), an N-manifold of degree
$1$ is canonically diffeomorphic to a split N-manifold of degree 1.\medskip

\item Without a choice of splittings, an
N-manifold $M=(M_0,\cA)$ of degree $n$ gives rise to a
filtration $$\cA_0\subset \cA_1\subset \cA_2\subset\ldots\subset
\cA_n=\cA\;,$$which implements a tower of fibrations
{$$M_0\leftarrow M_1\leftarrow M_2\leftarrow\ldots\leftarrow
M_n=M\;,$$}see \cite{Roy02}. Here, $M_1=E_{-1}[1]$, where $E_{-1}\to M_0$ is a smooth vector bundle, whereas $M_i\rightarrow M_{i-1}$, $i> 1$, is only a fibration. Each $M_i$ is an N-manifold of degree $i$.\end{itemize}

\noindent{\bf Example 2.} In case of the N-manifold $M=T^*[2]T[1]M_0$ associated to the standard Courant algebroid, we have the tower $$M_0\leftarrow T[1]M_0\leftarrow T^*[2]T[1]M_0\;.$$ Clearly, the N-manifold $M=M_2$ is not canonically of the type $E_{-1}[1]\oplus E_{-2}[2]$: it is only noncanonically split. On the other hand, it is even diffeomorphic, as NQ-manifold, to a split NQ-manifold, see \cite{SZ11}[Theorem 3.4].

\subsection{Examples of canonically split N-manifolds}

Split N-manifolds appear naturally in connection with the integration problem of exact Courant algebroids.

Let us be more precise. A representation of a Lie algebroid $(A,\ell_2,\zr)$ on a vector bundle $E$ over the same base, say $M$, is defined very naturally as a Lie algebroid morphism $\zd: A\to \cA(E),$ where $\cA(E)$ is the Atiyah algebroid associated to $E$ \cite{CM08}. In other words, $\zd$ is a $\Ci(M)$-linear map from $\zG(A)\to \zG(\cA(E))\simeq \op{Der}(\zG(E))$ (where the latter module is the module of derivative endomorphisms of $\zG(E)$), which respects the anchors and the Lie brackets. When viewing $\zd$ as a map $\zd:\zG(A)\times\zG(E)\to \zG(E)$, we can interpret it as an $A$-connection on $E$ with vanishing curvature (the anchor condition is automatically verified). It is well-known that a $TM$-connection on $E$ allows to extend the de Rham operator to $E$-valued differential forms $\zG(\w T^*M\otimes E)$. This generalization verifies the Leibniz rule with respect to the tensor product and it squares to 0 if and only if the connection is flat. Similarly, the flat $A$-connection $\zd$ on $E$, i.e. the representation $\zd$ on $E$ of the Lie algebroid $A$, can be seen as a degree 1 operator $d^\zd$ on $E$-valued bundle forms $$\zG(\w A^*\otimes E)\;,$$which is a derivation and squares to 0, i.e as a kind of homological vector field.

Many concepts -- even involved ones -- are very natural in this supergeometric setting. For instance, to obtain the notion of representation up to homotopy of a Lie algebroid, it suffices to replace in the latter algebraic definition of an ordinary Lie algebroid representation, the vector bundle $E$ by a graded vector bundle $E$ and to ask that the square zero derivation, say $D$, be of total degree 1. Representations up to homotopy, we denote them by $(E,D)$, provide the appropriate framework for the definition of the adjoint representation of a Lie algebroid and the interpretation of the Lie algebroid deformation cohomology \cite{CM08} as a cohomology associated to a representation. For more details we refer the reader to \cite{AC09}. When retranslated into the original geometric context (explicit transfers of this type can be found below), the preceding supergeometric definition means roughly that a representation up to homotopy of a Lie algebroid $A$, is a complex $(E,\p)$ of vector bundles endowed with an $A$-connection $\zd$ that is flat only up to homotopy: $R^{\zd}=-\p\zw_2$, the homotopy $\zw_2$ verifies a coherence law $d^{\zd}\zw_2=-\p\zw_3$ up to homotopy $\zw_3$, and so on.

To integrate an exact Courant algebroid \cite{SZ11}, Sheng and Zhu view this algebroid as an extension of the tangent bundle by its coadjoint representation up to homotopy and integrate that extension. The extensions of Lie algebroids $A$ by representations up to homotopy $(E,D)$, they consider, are twisted semidirect products of $A$ by $(E,D)$. More precisely, they are twists by cocycles of the representation up to homotopy induced by $D$ on $\odot sE^*$, where $s$ is the suspension operator and $E$ a graded vector bundle. Both, semidirect products and extensions, are canonically split N- and even NQ-manifolds. For instance, a semidirect product is a representation on $\odot sE^*$, i.e. a degree 1 square 0 derivation on $$\zG(\w A^*\otimes \odot sE^*)=\zG(\odot sA^*\0 \odot sE^*)=\zG(\odot (s^{-1}(A\oplus E))^*)\;,$$and therefore a split NQ-manifold.

\section{Geometry of Lie $n$-algebroids}

\subsection{Standard and homological gradings of split N-manifolds}\label{Gradings}

Let $\cM=(M,\cA)$ be an N-manifold of degree $n$, consider local coordinates in an open subset $U\subset M$, and denote by $u=(u^{\za})=(\xi_1^{j_1},\ldots,\xi_n^{j_n})$ the coordinates of nonzero degree. The weighted Euler
vector field $$\ze_U=\sum_{\za}\zvw(u^{\za})u^{\za}\p_{u^{\za}}\;,$$
where $\zvw(u^{\za})$ denotes the degree of $u^{\za}$, is
well-defined globally: $\ze_U=\ze|_U$, where $\ze$ is the degree-derivation of the graded algebra $\cA^{\bullet}=\oplus_k\cA^k$, i.e. for any $f\in\cA^k$, $\ze f=kf\in\cA^k.$ Denote now by $\op{Der}^{\bullet}\cA$ the sheaf of graded derivations of $\cA$, which is in particular a sheaf of graded Lie algebras -- we denote their brackets by $[-,-]$. It is well-known that the grading of these vector fields is captured by the degree zero interior derivation $[\ze,-]$ of $\op{Der}^{\bullet}\cA$: if $X\in\op{Der}^k\cA$, then $[\ze,X]=kX\in\op{Der}^k\cA$. The gradings of other geometric objects are encrypted similarly, see \cite{Roy02}. On the other hand, it is clear that the standard Euler vector field $\tilde\ze_U=\sum_{\za}u^{\za}\p_{u^\za},$ which encodes the local grading by the number of generators, is not global.\bigskip

\noindent{\bf Example 4.} Consider for instance $\cM=T^*[2]T[1]M$,
with coordinates $q^j,\xi^j,p_j,\zy_j$, where $\xi^j$ is thought
of as $dq^j$, $p_j$ as $\p_{q^j}$, and $\zy_j$ as $\p_{\xi^j}$, so
that they are of degree $0,1,2,1$, respectively. A simple example, e.g. the coordinate transformation $q=Q,\xi=Q\Xi+e^Q\Theta,p=Q\Xi\Theta +P,\theta=Q^2\Xi,$ for $\dim M=1$, allows to check the preceding claim by direct computation.\medskip

The situation is different for a split N-manifold $E[\cdot]=\oplus_{i=1}^n E_{-i}[i]$ over $M$. Its structure sheaf
$$\cA=\zG(\odot E^*)$$ $$=\Ci_M\oplus \zG(E^*_{-1})\oplus \lp
\w^2_{\Ci_M}\zG(E^*_{-1})\oplus\zG(E^*_{-2})\rp$$ $$\oplus \lp
\w^3_{\Ci_M}\zG(E^*_{-1})\oplus
\zG(E^*_{-1})\0_{\Ci_M}\zG(E^*_{-2})\oplus\zG(E^*_{-3})\rp\oplus\ldots\;$$clearly carries two gradings, the standard grading induced by $E$ (encoded by $\ze$) and the grading by the number of generators (encoded by $\tilde \ze$). We refer to the first (resp., second) degree as the {\it standard} (resp., {\it homological}) degree and write $$\cA^{\bullet}=\oplus_k\cA^k\quad (\text{resp.,}\quad ^{\bullet}\!\cA=\oplus_r\, ^{r}\!\cA)\;.$$ Both Euler fields, $\ze$ and $\tilde \ze$, are now global: $\tilde\ze$ is the degree-derivation of the graded algebra $^{\bullet}\!\cA=\oplus_r\, ^{r}\!\cA$, i.e. its eigenvectors with eigenvalue $r$ are the elements of $^{r}\!\cA$.\bigskip

\noindent{\bf Remark 2.} Split supermanifolds are not a full subcategory of supermanifolds: their morphisms respect the additional grading in the structure sheaf. For N-manifolds the situation is similar. N-manifolds are nonlinear generalizations of vector bundles \cite{Vor10}, in the sense that a coordinate transformation may contain nonlinear terms. If we denote, e.g. in degree $n=2$, the transformation by $(x^i,\xi^a,\zh^{\frak a})\leftrightarrow (y^i,\theta^a,\tau^{\frak a})$, this means that the general form of $\zt^{\frak b}$ is $$\zt^{\frak b}=\zt^{\frak b}_{ab}(x)\xi^a\xi^b+\zt^{\frak b}_{\frak a}(x)\zh^{\frak a}\;.$$ In the case of split N-manifolds, morphisms respect not only the standard degree, but also the homological one: $$y^{j}=y^{j}(x),\;\theta^b=\theta^b_a(x)\xi^a,\;\tau^{\frak b}=\tau^{\frak b}_{\frak a}(x)\zh^{\frak a}\;$$and
\be x^{i}=x^{i}(y),\;\xi^a=\xi^a_b(y)\theta^b,\;\zh^{\frak a}=\zh^{\frak a}_{\frak b}(y)\tau^{\frak b}\;.\label{TransFunctNMfd}\ee Invariance of $\ze_U$ and $\tilde\ze_U$ is now also easily checked by direct computation.\bigskip

Due to the additional grading in the structure sheaf $\cA$, vector fields of a split N-manifold $(M,\cA)$, i.e. sections of the sheaf
$\op{Der}^{\bullet}\cA$ of graded derivations of $\cA$, acquire a homological degree as well: we say that $X\in\op{Der}^{\ell}\cA$ has homological degree $s$ and we write $X\in\,^s\!\op{Der}^{\ell}\cA$, if $X(^r\!\cA)\subset\, ^{r+s}\!\cA$. This condition is equivalent to the requirement that $[\tilde\ze,X]=sX$. Indeed, $\tilde\ze\in\op{Der}^{0}\cA$ and if $f\in\, ^r\cA$, then
$$[\tilde \ze,X]f=\tilde\ze X f-X\tilde\ze f=\tilde\ze X f-rXf\;.$$ Hence, the announced equivalence.

\begin{prop}\label{BigradDer} The sheaf of vector fields of a split N-manifold $(M,\cA)$ of degree $n$ is bigraded, i.e. $$\op{Der}^{\bullet}\cA=\oplus_{\ell \ge -n}\op{Der}^{\ell}\cA\quad\text{and}\quad \op{Der}^{\ell}\cA=\oplus_{s\ge -1}\,^s\!\op{Der}^{\ell}\cA\;.$$\end{prop}

\begin{proof} The first part of the claim is obvious. Let now $X\in\op{Der}^{\ell}\cA$, consider -- to simplify notations -- the case $n=2$, and denote, as above, local coordinates in $U\subset M$ by $v=(x,\xi,\zh)$ and in $V\subset M$ by $w=(y,\theta,\zt)$, where superscripts are understood. The vector field $X$ locally reads \be X|_U=\sum_s\left[ f_s(v)\p_{x}+g_{s+1}(v)\p_{\xi}+h_{s+1}(v)\p_{\zh}\right]\label{CoordFormU}\ee and \be X|_V=\sum_s\left[ F_s(w)\p_{y}+G_{s+1}(w)\p_{\theta}+H_{s+1}(w)\p_{\zt}\right]\;,\label{CoordFormV}\ee where the sum over $s$ refers to the homological grading $\cA=\oplus_s\,^s\!\cA$ and where the degree of $\p_{x}$ (resp., $\p_{\xi}$, $\p_{\zh}$) is 0 (resp., $-1$, $-1$) with respect to the homological degree (and similarly for $\p_{y},\p_{\theta}, \p_{\zt})$. In view of (\ref{TransFunctNMfd}), we get $$\p_{x}=\p_{x}{y}|_{x=x(y)}\p_{y}+\p_{x}\theta|_{x=x(y)}
\xi(y)\theta\p_{\theta}+\p_{x}{\tau}|_{x=x(y)}\zh(y)\tau\p_{\tau}=:A(y)\p_y+B(y)\theta\p_{\theta}+C(y)\tau\p_{\zt}\;,$$ $$\p_{\xi}=:D(y)\p_{\theta}\quad\text{and}\quad \p_{\zh}=:E(y)\p_{\tau}\;.$$ Hence, on $U\cap V$,
\be X|_{U}=\sum_s\left[ f_s(w)A(y)\p_y+(f_{s}(w)B(y)\theta+g_{s+1}(w)D(y))\p_{\theta}\right.$$ $$\left.+(f_{s}(w)C(y)\tau +h_{s+1}(w)E(y))\p_{\zt}\right]\;.\label{CoordFormUV}\ee It follows that the coefficients of $\p_y,\p_{\theta},\p_{\zt}$ and therefore the brackets in (\ref{CoordFormV}) and (\ref{CoordFormUV}) coincide on $U\cap V$. This means that the brackets in (\ref{CoordFormU}) and (\ref{CoordFormV}), which are elements of $^s\!\op{Der}^{\ell}\cA(U)$ and $^s\!\op{Der}^{\ell}\cA(V)$, respectively, are the restrictions of a global vector field $X^s\in\,^s\!\op{Der}^{\ell}\cA$. Eventually, $X=\sum_sX^s$ and this decomposition is obviously unique.\end{proof}

\subsection{$\Z$-Graded symmetric tensor coalgebras}

This section provides information that will be needed later on. We begin with some results on the graded symmetric tensor coalgebra.\medskip

If $V$ is a graded vector space, we denote by $\bar \odot V=\oplus_{r\ge 1}\odot^r V$ the reduced graded symmetric tensor algebra over $V$. If $W$ denotes a graded vector space as well, any linear map $f\in\op{Hom}(\bar\odot V, \bar\odot W)$ is defined by its `restrictions' $f_{rs}\in \op{Hom}(\odot^rV,\odot^{s}W),$ $r,s\ge 1.$ The graded symmetric product $$f_{r_1s_1}\odot g_{r_2s_2}\in\op{Hom}(\odot^{r_1+r_2}V,\odot^{s_1+s_2}W)$$
is defined, for a homogeneous $g_{r_2s_2}$ -- of degree $g$ with respect to the standard grading --, by
$$({f_{r_1s_1}\odot g_{r_2s_2})(v_1,\ldots,v_{r_1+r_2})}=$$
\be\sum_{\zs\in\op{Sh(r_1,r_2)}}
(-1)^{g(v_{\zs_1}+\ldots+v_{\zs_{r_1}})}\ze(\zs)\;f_{r_1s_1}(v_{\zs_1},\ldots,v_{\zs_{r_1}})\odot
g_{r_2s_2}(v_{\zs_{r_1+1}},\ldots,v_{\zs_{r_1+r_2}})\;,\label{WedgeHom}\ee
where the $v_k$ are homogeneous elements of $V$ of degree denoted by $v_k$ as well, and where $\ze(\zs)$ is the Koszul sign. If $f,g$ are homogeneous, we have of course
$$g\odot f=(-1)^{fg}f\odot g\;.$$

The graded symmetric tensor product of linear maps is needed to construct coderivations and cohomomorphisms of graded symmetric tensor coalgebras from their corestrictions. The reduced graded symmetric coalgebra on a graded vector space $V$ is made up by the tensor space $\bar\odot V=\oplus_{r\ge 1}\odot^rV$ endowed with the coproduct $$\zD(v_1\odot\ldots\odot v_r):=\sum_{k=1}^{r-1}\sum_{\zs\in\op{Sh}(k,r-k)}\ze(\zs)(v_{\zs_1}\odot\ldots\odot v_{\zs_k})\otimes(v_{\zs_{k+1}}\odot\ldots\odot v_{\zs_r})\;,$$with self-explaining notations. We denote the symmetric coalgebra on $V$ by $\bar\odot^c V$.

A coderivation $\zd:\bar\odot^cV\to\bar\odot^cV$ (resp., a cohomomorphism $\zf:\bar\odot^cV\to \bar\odot^cW$) is completely defined by its corestrictions $\zd_r:\odot^rV\to V$ (resp., $\zf_r:\odot^rV\to W$), $r\ge 1$: \be{\label{CoDerWedge}}\zd(v_1\odot\ldots\odot v_r)=\sum_{k=1}^r(\zd_k\odot\op{id}_{r-k})(v_1\odot\ldots\odot v_r)\;,\ee (resp., \be\label{CohomoWedge}\zf(v_1\odot\ldots\odot v_r)=\sum_{s=1}^r\frac{1}{s!}\sum_{\tiny\begin{array}{c}r_1+\ldots+r_s=r\\r_i\neq 0\end{array}} (\zf_{r_1}\odot\ldots\odot\zf_{r_s})(v_1\odot\ldots\odot v_r)\;.)\ee These results are just a matter of computation.\bigskip

The next facts about the tensor power of the suspension operator will be useful. If $E=\oplus_{i=1}^nE_{-i}$ is as usually a graded vector bundle and if $s:\zG(E_{-k})\to\zG((sE)_{-k+1})$ denotes the shift operator, the assignment
$$\zG(E_{-k})^{\times 2}\ni(X_1,X_2)\mapsto
(-1)^{X_1} s X_1\boxdot s X_2\in \boxdot^2_{\Ci_M}\zG((sE)_{-k+1})$$is
graded symmetric and $\Ci_M$-bilinear, and thus defines a $\Ci_M$-linear map on
$\odot_{\Ci_M}^2\zG$ $(E_{-k})$. More generally, for $a_1\le \ldots \le a_i$, set $$s^i:\zG(E_{-a_1})\odot_{(\Ci_M)}\ldots\odot_{(\Ci_M)}\zG(E_{-a_i})\ni X_{1}\odot\ldots\odot X_i\mapsto$$ $$(-1)^{\sum_j (i-j)X_j}sX_{1}\boxdot\ldots\boxdot sX_i\in \zG((sE)_{-a_1+1})\boxdot_{(\Ci_M)}\ldots\boxdot_{(\Ci_M)}\zG((sE)_{-a_i+1})\;,$$where the tensor products in the source and target spaces are taken, either over $\Ci_M$, or over $\R$. The inverse of $s^i$ is given by $$(s^i)^{-1}=(-1)^{i(i-1)/2}(s^{-1})^i\;,$$where $s^{-1}$ is the desuspension operator.\bigskip

Just as for Lie infinity algebras, we will find two variants of the definition of Lie infinity algebroids. It is important to first fully understand the purely algebraic situation. Let $\ell'\in\op{Codiff}^1(\bar\odot^cV)$ be a degree 1 codifferential of the reduced graded symmetric tensor coalgebra of a desuspended graded vector space $V=s^{-1}W$. The condition $\ell'^2=0$ is satisfied if and only if the projection $\op{pr}_1$ onto $V$ of the restriction to $\odot^rV$ of $\ell'^2$ vanishes for all $r\ge 1$, i.e. if, see (\ref{CoDerWedge}), $$\op{pr}_1\sum_{i=1}^r\sum_{j=1}^{r-i+1}(\ell'_j\odot \op{id}_{r-i-j+1})(\ell'_i\odot\op{id}_{r-i})=\sum_{i+j=r+1}\ell'_j(\ell'_i\odot\op{id}_{r-i})=0,\; \text{for all } r\ge 1\;,$$where $\ell'_i:\odot^iV\to V$ is the $i$-th corestriction of $\ell'$ (it is of course of degree 1). This structure was discovered by Voronov \cite{Vor05} via derived brackets under the name of $L_{\infty}$-antialgebra.

\begin{defi}\label{LieInftyAntiAlg} An \emph{$L_{\infty}$-antialgebra} structure on a graded vector space $V$ is made up by
graded symmetric multilinear maps $\ell_i':V^{\times
i}\to V$, $i\ge 1,$ of degree 1, which verify the conditions
\be\sum_{i+j=r+1}\sum_{\zs\in\op{Sh}(i,j-1)}\ze(\zs)\ell'_{j}(\ell_i'(v_{\zs_1},\ldots,v_{\zs_i}),v_{\zs_{i+1}},\ldots,v_{\zs_r})=0\;,\label{VLInfty}\ee for all homogeneous $v_k\in V$ and all $r\ge 1$. \end{defi}

The definition of an $L_{\infty}$-algebra is similar, except that the multilinear maps are of degree $2-i$ and graded antisymmetric, and that the sign $\ze(\zs)$ in (\ref{VLInfty}) is replaced by $(-1)^{i(j-1)}\op{sign}(\zs)\ze(\zs)$, where $\op{sign}(\zs)$ denotes the signature. See \cite{LS93} and Definition \ref{DefLieNAld}, Equation (\ref{LSLInfty}) of the present text. \medskip

\begin{prop}\label{LSVLinfty} An $L_{\infty}$-antialgebra structure $\ell'_i$, $i\ge 1$, on $V=s^{-1}W$ induces an $L_{\infty}$-algebra structure $\ell_i:=s\;\ell'_i\;(s^{-1})^i$ on $W$, and, vice versa, an $L_{\infty}$-algebra structure $\ell_i$, $i\ge 1$, on $W$ implements an $L_{\infty}$-antialgebra structure $\ell'_i:=(-1)^{i(i-1)/2}s^{-1}\;\ell_i\;s^i$ on $V$.\end{prop}

The result is just a reformulation of the construction of a Lie infinity algebra from a codifferential. Let us nevertheless emphasize that the sign in the preceding correspondence is important (and can for instance not be transferred from the definition of $\ell_i'$ to that of $\ell_i$).

\subsection{Higher derived brackets construction of Lie $n$-algebroids}


We already mentioned that Lie algebroids are 1:1 with Q-manifolds. In this section we extend this correspondence to Lie $n$-algebroids and NQ-manifolds of degree $n$. Many authors consider Lie $n$-algebroids, but actually just mean NQ-manifolds \cite{Sev01}: it is only in 2011 that Sheng and Zhu defined split Lie $n$-algebroids by means of anchors and brackets \cite{SZ11}. They mention the bijection with split NQ-manifolds, but give no proof. It turned out that the latter is highly technical. Below, we provide two possible approaches to this correspondence. In particular, we prove that the orders of the brackets (as differential operators) of a Lie $n$-algebroid suggested in \cite{SZ11} are the only possible ones. For algebroids with generalized anchors, see \cite{GKP11}.\medskip

\noindent{\bf Remark 3.} Many concepts of Lie $n$-algebras appear in the literature. Lie $n$-algebras are in principle specific linear $(n-1)$-categories. But the term `Lie $n$-algebras' often also refers to $n$-ary Lie algebras and to $n$-term Lie infinity algebras. However, whereas Lie 2-algebras in the categorical sense and 2-term Lie infinity algebras are the objects of two 2-equivalent 2-categories \cite{BC04}, `categorical' Lie 3-algebras are (in 1:1 correspondence with) quite particular 3-term Lie infinity algebras (their bilinear and trilinear maps have to vanish in degree (1,1) and in total degree 1, respectively). The main reason for these complications is that the map $$\boxtimes:L\times L'\mapsto L\boxtimes L'\;,$$ where $\boxtimes$ denotes the monoidal structure of the category {\tt Vect} $n$-{\tt Cat} of linear $n$-categories, is not a bilinear $n$-functor \cite{KMP11}. Nevertheless, when speaking about a Lie $n$-algebra (resp., algebroid), we mean in this text an $n$-term Lie infinity algebra (resp., algebroid).\medskip

One of the possible approaches to split Lie $n$-algebroids is Voronov's higher derived brackets construction \cite{Vor05}, which we now briefly
recall. Let $L$ be a Lie superalgebra with bracket $[-,-]$ and let $P\in\op{End}L$ be a
projector, such that $V=P(L)$ be an abelian Lie subalgebra and
$$P[\ell,\ell']=P[P\ell, \ell']+P[\ell,P\ell']\;,$$for any $\ell,\ell'\in
L.$ The latter condition is just a convenient way to say that
$\op{Ker} P$ is a Lie subalgebra as well. Let us mention that this
setup implies that $L=V\oplus K,$ $K=\op{Ker}P$.
Consider now an odd derivation $D\in\op{Der}L$ that respects $K$,
i.e. $DK\subset K$, and construct higher derived brackets on $V$:
$$\{v_1,\ldots,v_k\}_D:=P[\ldots[[Dv_1,v_2],v_3],\ldots,
v_k]\;.$$If $D^2(V)=0,$ this sequence of $k$-ary brackets, $k\ge
1$, defines a Lie infinity antialgebra structure on $V$ and induces a Lie infinity algebra structure on $sV$.\medskip

Next we consider a split N-manifold $E[\cdot]=\oplus_{i=1}^n E_{-i}[i]$ of
degree $n$ over a base manifold $M$. The interior product of an element of
$\cA=\zG(\odot E^*)=\odot_{\Ci(M)}\zG(E^*)$ by $X_j\in
\zG(E_{-j})$ is defined by $0$ on $f\in\Ci(M)$, on the other generators
$\zw_k\in\zG(E_{-k}^*)$ by
$$i_{X_j}\zw_k=\zd_{jk}(-1)^{jk}\zw_k(X_j)\in
\cA^{k-j}\;,$$where $\zd_{jk}$ is Kronecker's symbol, and it is extended
to the whole graded symmetric algebra $\cA^{\bullet}$ as a derivation of degree $-j$, i.e. by $$i_{X_j}(S\odot
T)=(i_{X_j}S)\odot T+(-1)^{j\ell}S\odot (i_{X_j}T)\;,$$where
$S\in\cA^{\ell}$ and $T\in\cA$. It is
clear that we thus assign to any $X\in\zG(E)$ a unique
$i_X=\sum_ji_{X_j}\in\, ^{-1}\!\op{Der}\cA$. As
usual:

\begin{lem}\label{-1Der} The sections in $\zG(E)$ are exactly the derivations in $^{-1}\!\op{Der}\cA$.\end{lem}

\begin{proof} Let $\zd\in\, ^{-1}\!\op{Der}\cA$. Locally,
in coordinates $v=(x,u)=(x^{i},u^{\za})$ over
$U\subset M$, where the $x^i$ are the base coordinates and the
$u^{\za}$ the shifted fiber coordinates, this derivation reads
$$\zd|_U=\sum_{\za} f^{\za}(x)\p_{u^{\za}}\in\,^{-1}\!\op{Der}\cA(U)\;.$$ The coordinates
$u^{\za}$ of the fibers of the $E_{-i}$ can be interpreted as frames of the $E_{-i}^*$ over $U$. Let now $u_{\za}$ be
the dual frames of the $E_{-i}$ over $U$, $u_{\za}(u^{\zb})=\zd^{\zb}_{\za}$, and consider $$X_U:=\sum_{\za}f^{\za}(x)u_{\za}\in\zG(U,E)\quad\text{and}\quad
i_{X_U}\in\,^{-1}\!\op{Der}\cA(U)\;.$$Since the actions
of the derivations $\p_{u^{\za}}$ and $i_{u_{\za}}$ on the generators of $\cA(U)$ coincide, we have $\zd|_U=i_{X_U}$. It follows
that the local sections $X_U\in\zG(U,E)$ defined over different
chart domains $U$ coincide in $U\cap V$ and thus define a global
section $X\in\zG(E)$, such that $X|_U=X_U$. Finally, we obtain
$\zd=i_X$.\end{proof}

It is now quite easy to see which algebroid structure is encoded in the data of a split NQ-manifold. Let $E[\cdot]=\oplus_i E_{-i}[i]$ be a split N-manifold of degree $n\ge 1$ and let $Q\in\op{Der}^1\cA$ be a homological vector field, $Q^2={\small \frac{1}{2}}[Q,Q]=0\;.$ For instance, in the case $n=2$ and in local
coordinates $(x^i,\xi^a,\zh^{\frak a})$, such a derivation is of the type \be\label{QDEG1}Q=\sum
f_a^{i}(x)\xi^a\p_{x^{i}}+\sum (g^a_{bc}(x)\xi^b\xi^c
+h^a_{\frak a}(x)\zh^{\frak a})\p_{\xi^a}+\sum (k^{\frak a}_{abc}(x)\xi^a\xi^b\xi^c
+\ell_{a\frak b}^{\frak a}(x)\xi^a\zh^{\frak b})\p_{\zh^{\frak a}}\;;\ee it contains terms
of homological degrees $0,1,2$. The considered NQ-manifold induces all the data required by the setup of the higher derived brackets method. Indeed, let $$L=\op{Der}^{\bullet}\cA=\oplus_{\ell\ge -n}\op{Der}^{\ell}\cA=\oplus_{\ell\ge -n}\oplus_{s\ge -1}\,^{s}\!\op{Der}^{\ell}\cA\;$$ be the graded Lie algebra of
derivations of $\cA^{\bullet}=\oplus_k\cA^k$ with bracket denoted by $[-,-]$. Observe that here we consider the standard grading, but that, see Proposition \ref{BigradDer}, that the space $L$ is also graded by the homological degree. The graded commutator $[-,-]$ respects this homological degree as well. In fact, for $\zd\in\,^{r}\!\op{Der}\cA$ and
$\zd'\in\,^{s}\!\op{Der}\cA$, we have
$$[\tilde\ze,[\zd,\zd']]=[[\tilde\ze,\zd],\zd']+[\zd,[\tilde\ze,\zd']]=(r+s)[\zd,\zd']\;,$$where $\tilde\ze\in\op{Der}^0\cA$ is the Euler field that encodes the homological degree of `functions' and `vector fields', see Section \ref{Gradings}. Denote now by $P:L\to L$ the projector
onto $^{-1}\!\op{Der}\cA$. Hence, $V:=P(L)=\,
^{-1}\!\op{Der}\cA\simeq \zG(E)$, see Lemma \ref{-1Der}, is an abelian Lie
subalgebra of $L$. Moreover, if $\zd=\sum_{r\ge -1}\,^r\!\zd\in L$
and $\zd'=\sum_{s\ge -1}\,^{s}\!\zd'\in L$, we have
$$P[\zd,\zd']=[\,^{-1}\!\zd,\,^0\!\zd']+[\,^0\!\zd,\,^{-1}\!\zd']=P[P\zd,\zd']+P[\zd,P\zd']\;.$$Let
now $D:=[Q,-]$ be the interior degree 1 derivation of $L$ induced
by $Q$. It respects the kernel $K:=\op{Ker}P=\oplus_{r\ge 0}
\,^r\!\op{Der}\cA$. Indeed,
$Q\in\op{Der}^1\cA=\oplus_{s\ge -1}\,^s\!\op{Der}^1\cA$ reads
$Q=\sum_{s=0}^n\,^{s}\!Q$, see e.g. Equation (\ref{QDEG1}), and,
if $\zk=\sum_{r\ge 0}\,^r\!\zk\in K$, we get
$D\zk=\sum_{r}\sum_s[\,^{s}\!Q,\,^r\!\zk]\in K.$ As
$D^2=[Q,[Q,-]]=0$, the higher derived brackets
$$\ell'_k(X_1,\ldots,X_k)=P[\ldots[[Q,X_1],X_2],\ldots,
X_k]\;$$provide a $L_{\infty}$-antialgebra structure on $V=\zG(E).$ These $k$-ary brackets are actually given by
\be\label{DefQ2DerBrack}\ell'_k(X_1,\ldots,X_k)=P\sum_{s=0}^n[\ldots[[\,^{s}\!Q,X_1],X_2],\ldots, X_k]=[\ldots[[\,^{k-1}\!Q,X_1],X_2],\ldots,
X_k]\;,\ee for $1\le k\le n+1$, and they vanish otherwise. In view of Proposition \ref{LSVLinfty}, the brackets $\ell_k:=s\,\ell'_k\,(s^{-1})^k$ endow $\zG(sE)$ with a Lie infinity structure.

To discover the algebroid structure encrypted in the homological vector field of an N-manifold, it remains to extract the information contained in the action of $Q$ on the generators of degree 0, i.e. on $\cA^0=\,^0\!\cA=\Ci(M)$. Since, $^{s}\!Q:\Ci(M)\to \zG(E^*_{-1})\cap\, ^{s}\!\cA\;,$ the derivation $^s\!Q$ vanishes on functions, if $s\neq 1$, whereas $$^{1}\!Q :\Ci(M)\to \zG(E^*_{-1})\;.$$ For any $X\in\zG(E_{-1})$ and any $f\in\Ci(M)$, we set \be\label{DefQ1DerBrack}\zr'(X)f:=[^1\!Q,X]f=i_X\,^1\!Q f=-(^1\!Q f)(X)\in\Ci(M)\;.\ee As $[^1\!Q,X]\in\,^0\!\op{Der}\cA$ restricts to a derivation $[^1\!Q,X]\in \op{Der}\cA^0$, $\zr'$ is a $\Ci(M)$-linear map $\zr':\zG(E_{-1})\to \zG(TM)$ and can thus be viewed as a bundle map $\zr':E_{-1}\to TM$. Moreover, if $X_j\in\zG(E)$ and $f\in \Ci(M)$, the bracket $$\ell'_k(X_1,\ldots,fX_j,\ldots,X_k)=[\ldots[\ldots[\,^{k-1}\!Q,
X_1],\ldots,fX_j],\ldots, X_k]\;$$can be computed as follows. It
is easily seen that $$[\ldots[\,^{k-1}\!Q,
X_1],\ldots,fX_j]=\lp[\ldots[\,^{k-1}\!Q, X_1],\ldots,X_{j-1}]f\rp
\odot i_{X_j}+f[\ldots[\,^{k-1}\!Q, X_1],\ldots,X_j]$$and
$$[\ldots[\,^{k-1}\!Q, X_1],\ldots,X_{j-1}]f=\pm i_{X_{j-1}}\ldots
i_{X_1}\,^{k-1}Qf\;.$$If $k=2$ and $j=2$ (for $j=1$ it suffices to
use the symmetry of the bracket), we thus find
$$\ell'_2(X_1,fX_2)=(i_{X_1}\,^1\!Qf)\;X_2+f\ell'_2(X_1,X_2)\;.$$The first term
is nonzero only if $X_1\in\zG(E_{-1})$, in which case it reads $(\zr'(X_1)f)X_2$. If $k\neq 2,$ we get
$$\ell'_k(X_1,\ldots,fX_j,\ldots,X_k)=f\ell'_k(X_1,\ldots,X_j,\ldots,X_k)\,,$$as $X_{m}\simeq i_{X_m}$ is $\Ci(M)$-linear. If we set now $\zr=\zr's^{-1}$, we obtain a bundle map $\zr:(sE)_0\to TM$, such that the brackets $\ell_i=s\,\ell_i'\,(s^{-1})^i$, $i\neq 2$, are $\Ci(M)$-multilinear, whereas $\ell_2$ is $\Ci(M)$-bilinear if both arguments are elements of $\zG((sE)_{-i})$, $i\neq 0$, and verifies $$\ell_2(x_0,fx)=f\ell_2(x_0,x)+(\zr(x_0)f)x\;,$$for any $x_0\in\zG((sE)_0)$, $x\in\zG(sE)$, and $f\in\Ci(M)$.\medskip

\noindent{\bf Remark 4.} We thus recover the concept of split Lie $n$-algebroid introduced by Sheng and Zhu in \cite{SZ11}. In fact, we proved that to any split NQ-manifold of degree $n$ is associated a split Lie $n$-algebroid.

\begin{defi}\label{DefLieNAld} A \emph{split Lie $n$-algebroid}, $n\ge 1$, is a graded vector bundle $L=\oplus_{i=0}^{n-1}L_{-i}$ over a smooth manifold $M$, together with a bundle map $\zr:L_0\to TM$ and graded antisymmetric $i$-linear brackets $\ell_i:\zG(L)^{\times i}\to\zG(L)$, $i\in\{1,\ldots,n+1\}$, of degree $2-i$, such that \begin{itemize}\item for any $r\ge 1$, \be\sum_{i+j=r+1}\sum_{\zs\in\op{Sh}(i,j-1)}(-1)^{i(j-1)}\chi(\zs)\ell_{j}(\ell_i(x_{\zs_1},\ldots,x_{\zs_i}),x_{\zs_{i+1}},\ldots,x_{\zs_r})=0\;,\label{LSLInfty}\ee where $x_1,\ldots,x_r\in \zG(L)$, where $\op{Sh}(i,j-1)$ denotes the set of $(i,j-1)$-shuffles, and where
$\chi(\zs)=\op{sign}(\zs)\ze(\zs)$ is the signature multiplied by
the Koszul sign with respect to the grading of $\zG(L)$, \item for $i\neq 2$, $\ell_i$ is $\Ci(M)$-multilinear, whereas $\ell_2$ is $\Ci(M)$-bilinear if both arguments belong to $\zG(L_{-i})$, $i\neq 0$, and verifies $$\ell_2(x_0,fx)=f\ell_2(x_0,x)+(\zr(x_0)f)\,x\;,$$ for any $x_0\in\zG(L_0)$, $x\in\zG(L)$, and $f\in\Ci(M)$.\end{itemize} \end{defi}

Observe that, since the graded vector bundle underlying a split Lie $n$-algebroid is concentrated in degrees $0,\ldots,-n+1$, the Lie infinity algebra conditions (\ref{LSLInfty}) are nontrivial only for $1\le r\le n+2$. Indeed, if $r\ge n+3$, the degree of the {\small LHS}-terms is at most $2-i+2-j=3-r\le -n$, so that all the terms vanish.\medskip

\noindent{\bf Examples 5.} A Lie 1-algebroid is a Lie algebroid in the usual sense. Indeed, $\ell_1$ vanishes, as it is of degree 1, and the Lie infinity algebra conditions reduce to the Jacobi identity. Further, a Lie $n$-algebroid over a point is exactly a Lie $n$-algebra, i.e. an $n$-term Lie infinity algebra.\medskip

\noindent{\bf Remark 5.} As in the case of Lie infinity algebras, there exists a notion of {\it Lie $n$-antialgebroid}. Such an antialgebroid is made up by a graded vector bundle $K=\oplus_{i=1}^nK_{-i}$ over a manifold $M$, a bundle map $\zr':K_{-1}\to TM$, and a $L_{\infty}$-antialgebra structure $\ell'_i$, $1\le 1\le n+1$, on $\zG(K)$, such that the $\ell_i'$ are $\Ci(M)$-multilinear for $i\neq 2$, whereas $\ell'_2$ is $\Ci(M)$-bilinear if both arguments belong to $\zG(K_{-i})$, $i\neq 1$, and verifies $$\ell_2'(x_{1},fx)=f\ell_2'(x_{1},x)+(\zr'(x_{1})f)x\;,$$ for all $x_{1}\in\zG(K_{-1}),$ $x\in\zG(K)$, and $f\in\Ci(M)$. To any Lie $n$-antialgebroid structure $\ell'_i,\zr'$ on $K$ corresponds a Lie $n$-algebroid structure $\ell_i=s\,\ell_i'\,(s^{-1})^i$, $\zr=\zr's^{-1}$ on $sK$, and vice versa.\medskip

\noindent{\bf Remark 6.} The orders of the brackets $\ell_i$, viewed as differential operators, are comprehensible from the above construction of a Lie $n$-algebroid by means of derived brackets. Only the binary bracket $\ell_2$ has an anchor, since $^s\!Q$ that implements $\ell_{s+1}$ vanishes on functions for $s\neq 1$. The reader might prefer an explanation in local coordinates. Consider first the case $n=1$ of Lie algebroids. When dualizing the Lie bracket and the anchor in a local trivialization over a chart domain, we get the homological vector field
$$Q=\zr^{i}_a(x)\xi^a\p_{x^{i}}+C^a_{bc}(x)\xi^b\xi^c\p_{\xi^a}\;,$$where $x^{i}$ (resp., $\xi^a$) are the base coordinates (resp., the fiber coordinates), and where $\zr^{i}_a(x)$ (resp., $C^a_{bc}(x)$) are the components
(resp., the structure functions) of the anchor (resp., of the bracket). Note that {the anchor is thus encrypted in the terms in the $\p_{x^i}$}. If we pass to $n=2$ and choose local coordinates $(x^i,\xi^a,\zh^{\frak a})$, a homological vector field is, as mentioned above, of the type $$Q=\sum
f_a^{i}(x)\xi^a\p_{x^{i}}+\sum (g^a_{bc}(x)\xi^b\xi^c
+h^a_{\frak a}(x)\zh^{\frak a})\p_{\xi^a}+\sum (k^{\frak a}_{abc}(x)\xi^a\xi^b\xi^c
+\ell_{a\frak b}^{\frak a}(x)\xi^a\zh^{\frak b})\p_{\zh^{\frak a}}\;,$$so that only $^1\!Q$ that defines $\ell_2$ contains such terms.

\subsection{Categories of Lie $n$-algebroids and NQ-manifolds: comparison of objects}

\begin{theo}\label{MTheo2} {There is a 1:1 correspondence between split NQ-manifolds $(E,Q)$ of degree $n$ and split Lie
$n$-algebroids $(sE,(\ell_i)_i,\zr)$}.\end{theo}

\begin{proof} Let us first mention that the proof uses the characterization of tensor fields as function-valued function-multilinear maps.

Let now $\ell_1,\ldots,\ell_{n+1},\zr$ be a Lie $n$-algebroid structure on a graded vector bundle $sE=(sE)_0\oplus\ldots\oplus (sE)_{-n+1}$ over a manifold $M$. We will define on the degree $n$ split N-manifold $E[\cdot]=\oplus_{i=1}^nE_{-i}[i]$ a derivation $Q\in \op{Der}^1\cA$ (or, better, a global section of the sheaf of degree 1 derivations) that squares to 0. Remember that $\cA$ (here, the algebra of global sections of the function sheaf) is given by $$\cA=\zG(\odot E^*)$$ $$=\Ci(M)\oplus \zG(E^*_{-1})\oplus \lp
\zG(\odot^2E^*_{-1})\oplus\zG(E^*_{-2})\rp\oplus \lp
\zG(\odot^3E^*_{-1})\oplus
\zG(E^*_{-1}\odot E^*_{-2})\oplus\zG(E^*_{-3})\rp\oplus\ldots\;$$

We first define the derivation $Q$ on the generators $\zw_k\in\zG(E^*_{-k})$, $k\in\{1,\ldots,n\}$, and $\zw_0\in\Ci(M)$. We decompose $Q\zw_k\in\cA^{k+1}$ with respect to the homological grading $\cA^{k+1}=\oplus_{r=1}^{k+1}\;^r\!\cA^{k+1}$ given by the number of generators. The component in $^r\!\cA^{k+1}$ will be denoted by $Q^{k+1,r}\zw_k$. For instance, $$Q^{4,2}\zw_3\in\zG(E^*_{-1}\odot E^*_{-3})\oplus \zG(\odot^2E^*_{-2})\;.$$Hence, we have to define, for $X_1\in\zG(E_{-1})$ and $X_2\in\zG(E_{-3})$ or $X_1,X_2\in\zG(E_{-2})$, $$(Q^{4,2}\zw_3)(X_1,X_2)\in\Ci(M)\;,$$in a way that this function depend $\Ci(M)$-bilinearly and graded symmetrically on its arguments. More generally, we define $(Q^{k+1,r}\zw_k)(X_1,\ldots,X_r)\in\Ci(M)$, for all $X_j\in\zG(E_{-a_j})$ such that $\sum_ja_j=k+1.$ We set \be\label{DefQ0}\zr'=\zr s,\quad \ell_r'=(-1)^{r(r-1)/2}s^{-1}\ell_rs^r\;,\ee so that $\ell'_r,\zr'$ provide a Lie $n$-antialgebroid structure on $E$. We now define \be(Q^{1,1}\zw_0)(X_1)=-\zr'(X_1)\zw_0\in\Ci(M)\;\label{DefQ1}\ee and, for $k\in\{1,\ldots,n\}$, \be(Q^{k+1,r}\zw_k)(X_1,\ldots,X_r)=(-1)^{k}\zw_k(\ell'_r(X_1,\ldots,X_r))\in\Ci(M)\;,\label{DefQ2}\ee if $r\in\{1,3,\ldots,k+1\}$, and \be(Q^{k+1,2}\zw_k)(X_1,X_2)=(-1)^{k}\zw_k(\ell_2'(X_1,X_2))-(\zr'\odot\zw_k)(X_1,X_2)\in\Ci(M)\;,\label{DefQ3}\ee if $r=2$. The tensor product in the last equation is given by $$(\zr'\odot\zw_k)(X_1,X_2)=(-1)^k\zr'(X_1)\zw_k(X_2)+(-1)^{a_1a_2+k}\zr'(X_2)\zw_k(X_1)\in\Ci(M)\;.$$ Here (and in the following) we implicitly extend the anchor $\zr':\zG(E_{-1})\to \zG(TM)$ and $\zw_k:\zG(E_{-k})\to \Ci(M)$ by 0 to the whole module $\zG(E)$. 

Graded symmetry is obvious and $\Ci(M)$-multilinearity is nontrivial only for $r=2$. 
Since $$\ell'_2(X,fY)=f\ell'_2(X,Y)+(\zr'(X)f)Y\quad\text{and}\quad\ell'_2(fX,Y)=f\ell'_2(X,Y)+(-1)^X(\zr'(Y)f)X\;,$$for all $X,Y\in\zG(E)$, the function $(Q^{k+1,2}\zw_k)(X_1,X_2)$ is actually $\Ci(M)$-bilinear.


To define $Q$ on an arbitrary element $\zw=\sum_{k,s}\zw_{k,s}\in\zG(\odot E^*)=\cA=\oplus_{k,s}\, ^s\!\cA^k$,  we define the projections $Q^{k+1,r}\zw_{k,s}$ of $Q\zw_{k,s}\in\cA^{k+1}=\oplus_{r=1}^{k+1}\,^r\!\cA^{k+1}$ onto $^r\!\cA^{k+1}$.

\begin{defi}\label{LieNAldCE} Let $\ell_i,\zr$ be a Lie $n$-algebroid structure on $sE$ and let $\ell_i'=(-1)^{i(i-1)/2}s^{-1}\ell_i\,s^i, \zr'=\zr\,s$ be the associated Lie $n$-antialgebroid data on $E$. The derivation $Q\in\op{Der}^1\zG(\odot E^*)$, which is defined by \be Q^{k+1,r}\zw_{k,s}=(-1)^k\zw_{k,s}\circ (\ell'_{r-s+1}\odot\op{id}_{s-1})-\zr'\odot\zw_{k,s}\;,\label{LieNAldCEOper}\ee where $\op{id}_{s-1}(X_1,\ldots,X_{s-1})=X_1\odot\ldots\odot X_{s-1}$, is the \emph{Chevalley-Eilenberg differential} of the Lie $n$-algebroid $sE.$\end{defi}

Equation (\ref{LieNAldCEOper}) can be written more explicitly. For $k=0$, we get \be\label{DefQ1'}Q^{k+1,r}\zw_{k,s}=-\zr'\odot\zw_{k,s}\;,\ee
for $k\ge 1$, if $r\ge s$, $r\neq s+1$, \be\label{DefQ2'} Q^{k+1,r}\zw_{k,s}=(-1)^{k}\,\zw_{k,s}\circ(\ell'_{r-s+1}\odot \op{id}_{s-1})\;,\ee if $r=s+1$, \be\label{DefQ3'} Q^{k+1,r}\zw_{k,s}=(-1)^{k}\,\zw_{k,s}\circ(\ell'_2\odot\op{id}_{s-1})-\zr'\odot\zw_{k,s}\;,\ee and if $r<s$, \be\label{DefQ4'} Q^{k+1,r}\zw_{k,s}=0\;.\ee

Indeed, if $k=0$, then $s=0$ and $\op{id}_{s-1}=0$, so that (\ref{LieNAldCEOper}) reduces to (\ref{DefQ1'}). Equations (\ref{DefQ2'}) and (\ref{DefQ3'}) are clear as well, as the term $\zr'\odot\zw_{k,s}$ can be evaluated only on $s+1$ sections of $E$ (and must be interpreted as 0 on $r\neq s+1$ sections). For (\ref{DefQ4'}), it suffices to note that $\ell'_i=0$, if $i\le 0$.

Let us briefly comment on the definitions (\ref{DefQ1'})-(\ref{DefQ4'}). Equation (\ref{DefQ1'}) is just a reformulation of (\ref{DefQ1}). As for (\ref{DefQ2'}) and (\ref{DefQ3'}), note that the argument $\zw_{k,s}$ is an element of $^s\!\cA^k$ and more precisely, say, of $\zG(E^*_{-c_1}\odot\ldots\odot E^*_{-c_s})$, $\sum_jc_j=k$, (for example $\zG(E^*_{-1}\odot E^*_{-3})$). Its image $Q^{k+1,r}\zw_{k,s}\in\, ^r\!\cA^{k+1}$ must be evaluated on $X_j\in\zG(E_{-a_j})$, $j\in\{1,\ldots,r\}$, $\sum_ja_j=k+1$. The computation of $\ell'_{r-s+1}\odot \op{id}_{s-1}$ on these $X_j$ leads to terms each of which belongs to some $\zG(E_{-b_1})\odot\ldots\odot\zG(E_{-b_s})$, $\sum_jb_j=k$, since $\ell'_{r-s+1}$ has degree 1 (in the example, $\zG(E_{-1})\odot\zG(E_{-3})$ or $\zG(E_{-2})\odot\zG(E_{-2})$). In (\ref{DefQ2'}),(\ref{DefQ3'}) it is understood that the evaluation of $\zw_{k,s}$ on those terms that do not match is 0 by definition. Graded symmetry and $\Ci(M)$-multilinearity are again straightforwardly checked.

To make Definition \ref{LieNAldCE} meaningful (and to complete the proof), we still have to show that $Q$ is a derivation and that $Q^2=0$.

As concerns the derivation property, observe that it follows from (\ref{CoDerWedge}) and (\ref{LieNAldCEOper}) that, for $\zr'=0$, the endomorphism $Q$ is actually a derivation. However, the map $\zw_{k,s}\mapsto\zr'\odot\zw_{k,s}$ is a derivation as well. Indeed, let $\zh_{\ell,t}\in\, ^t\!\cA^{\ell}$ and note that the tensor products in $\zr'\odot (\zw_{k,s}\odot \zh_{\ell,t})$ are defined differently. The first one is defined by means of vector fields that act on functions (we will use the notation $L$), the second by means of products of functions (notation $\cdot$). When omitting the arguments $X\in\zG(E)$ and the subscripts, we can write $$\zr'\odot (\zw\odot \zh)=\sum\pm L_{\zr'}(\zw \cdot \zh)=\sum\pm (L_{\zr'}\zw)\cdot \zh +\sum\pm \zw\cdot (L_{\zr'}\zh)$$ $$=(\zr'\odot\zw)\odot\zh+(-1)^{k}\zw\odot(\zr'\odot\zh)\;.$$ Eventually, $Q\in\op{Der}^1\cA$.

Below, we will explain that the Chevalley-Eilenberg complex of a Lie $n$-algebroid, Definition \ref{LieNAldCE}, `reduces' for $n=1$ to the Lie algebroid de Rham complex.

We prove now that $Q^2=0$, which holds true if it holds on the generators $\zw_k\in\zG(E^*_{-k})$, $1\le k\le n$, and $\zw_0\in\Ci(M)$. To increase readability we work first up to sign. In the addendum to the proof, the interested reader can find the details about signs. It suffices to show that the sum \be\ S=\sum_{s=1}^{k+1}(Q^{k+2,r}Q^{k+1,s}\zw_k)(X_1,\ldots,X_r)\label{Sum}\ee vanishes, for all $X_j\in\zG(E_{-a_j})$, such that $\sum_ja_j=k+2$, and each $1\le r\le k+2.$ Ignoring the signs, we get $$S=\sum_{s=1}^{\op{inf}(k+1,r)}(Q^{k+1,s}\zw_k)((\ell'_{r-s+1}\odot\op{id}_{s-1})(X_1,\ldots,X_r))+(\zr'\odot (Q^{k+1,r-1}\zw_k))(X_1,\ldots,X_r)\;.$$ In the preceding sum, we can replace $\inf(k+1,r)$ by $r$, since $Q^{k+1,k+2}\zw_{k}=0.$ Setting $t=r-s+1$, we then obtain $$ S=\sum_{s+t=r+1}\sum_{\zs\in\op{Sh}(t,s-1)}(Q^{k+1,s}\zw_k)(\ell_t'(X_{\zs_1},\ldots,X_{\zs_t}), X_{\zs_{t+1}},\ldots, X_{\zs_r})$$ $$+ \sum_i\zr'(X_i)(Q^{k+1,r-1}\zw_k)(X_1,\ldots\hat{i}\ldots, X_r)\;$$ $$=\zw_k\lp\sum_{s+t=r+1}\sum_{\zs\in\op{Sh}(t,s-1)}\ell'_s(\ell_t'(X_{\zs_1},\ldots,X_{\zs_t}), X_{\zs_{t+1}},\ldots, X_{\zs_r})\rp$$ $$+\sum_{\zs\in\op{Sh}(r-1,1)}(\zr'\odot\zw_k)(\ell_{r-1}'(X_{\zs_1},\ldots,X_{\zs_{r-1}}), X_{\zs_r})$$ \be\label{Signs1}+\sum_i\zr'(X_i)\lp\zw_k(\ell'_{r-1}(X_1,\ldots\hat{i}\ldots, X_r))+\zd_{r,3}(\zr'\odot \zw_k)(X_1,\ldots\hat{i}\ldots, X_r)\rp\;,\ee where the first term vanishes in view of the $L_{\infty}$-antialgebra condition (\ref{VLInfty}). Hence, $$S=\sum_i\zr'(\ell_{r-1}'(X_1,\ldots\hat{i}\ldots,X_r))\zw_k(X_i)+\sum_i\zr'(X_i)\zw_k(\ell_{r-1}'(X_1,\ldots\hat{i}\ldots,X_r))$$ \be\label{Signs2}+\sum_i\zr'(X_i)\zw_k(\ell'_{r-1}(X_1,\ldots\hat{i}\ldots, X_r))+\zd_{r,3}\sum_i\zr'(X_i)\lp(\zr'\odot \zw_k)(X_1,\ldots\hat{i}\ldots, X_r)\rp\;.\ee When taking signs into account, we see that the second and third sums cancel out. If no $a_i$ is equal to $k$, the first and fourth sums vanish as well. Otherwise, $a_1+\ldots \hat{i}\ldots +a_r=2$ and $r=2$ or $r=3$. If $r=2$, the first sum reads \be\label{Signs3}\zr'(\ell'_1(X_2))\zw_k(X_1)+\zr'(\ell'_1(X_1))\zw_k(X_2)\;,\ee where $a_1=k$ and $a_2=2$ or vice versa. It thus suffices to show that \be\label{ReprCond1}\zr'\circ\ell_1'=0\ee on $\zG(E_{-2})$. This conclusion follows from the $L_{\infty}$-condition $$\ell'_1(\ell'_2(X,fY))+\ell'_2(\ell'_1(X),fY)+(-1)^X\ell'_2(X,\ell'_1(fY))=0\;,$$written for $X\in\zG(E_{-2})$. If $r=3$, the first and fourth sums exist, so that \be\label{Signs4}S=\zr'(\ell'_2(X_2,X_3))\zw_k(X_1)+\ldots + \zr'(X_2)\lp \zr'(X_1)\zw_k(X_3)+\zr'(X_3)\zw_k(X_1)\rp+\ldots\;,\ee where one of the $a_i$ is equal to $k$ and the two others equal to $1$. It is easily seen that it suffices to prove that \be\label{ReprCond2}\zr'(\ell'_2(X,Y))=\zr'(X)\zr'(Y)-\zr'(Y)\zr'(X)\;,\ee for all $X,Y\in\zG(E_{-1})$. The result is encoded in the $L_{\infty}$-condition for brackets $\ell'_i,\ell'_j$, $i+j=4$, written for $X,Y,fZ$, with $X,Y\in\zG(E_{-1})$. This is straightforwardly checked (we actually obtain the same property for $\zr$ and $\ell_2$). This completes the construction of an NQ-manifold from a Lie $n$-algebroid.\medskip

Conversely, we can construct a Lie $n$-algebroid from an NQ-manifold $(E,Q)$. Indeed, the definitions (\ref{DefQ1})-(\ref{DefQ3}) can easily be inverted. Equation (\ref{DefQ1}) defines the anchor $\zr'$ from $Q$. Let now $X_j\in\zG(E_{-a_j})$, $1\le j\le r$, set $k:=\sum a_j-1$, and let $\zw_k\in\zG(E^*_{-k}).$ Equation (\ref{DefQ2}) gives, for $r\neq 2$, $$\lp\ell'_r(X_1,\ldots,X_r)\rp(\zw_k)=(Q^{k+1,r}\zw_k)(X_1,\ldots,X_r)\;,$$ since $(-1)^{k(k+1)}=1$. Equation (\ref{DefQ3}) provides $\left(\ell'_2(X_1,X_2)\right)(\zw_k)$. Clearly, $\zr'$ coincides with the anchor (\ref{DefQ1DerBrack}), say $\zr''$, defined in the construction of a Lie $n$-algebroid via higher derived brackets. Moreover, if we denote the higher brackets (\ref{DefQ2DerBrack}) by $\ell''_r$, we have \be\label{Dual-DerBrack}\ell_r'=(-1)^r\ell_r''\;.\ee Indeed, when computing $$ \lp\ell_r''(X_1,\ldots,X_r)\rp(\zw_k)=[\ldots[[^{r-1}\!Q,X_1],X_2],\ldots,X_r](\zw_k)\;,$$ we get terms of the type $i_{X_{\zs_1}}\ldots i_{X_{\zs_j}}\,^{r-1}\!Q\, i_{X_{\zs_{j+1}}}\ldots i_{X_{\zs_r}}\zw_k$. However, if $j$ differs from $r$ and $r-1$, such a term vanishes. Even for $j=r-1$, it vanishes, except if $a_{\zs_r}=k,$ in which case we have $r=2$, since $\sum a_j=k+1$. If $r\neq 2$, the derived bracket $\ell''_r$ is given by a unique term. It suffices to compute the sign of this interior product and to insert the sections $X_j$ into $^{r-1}\!Q\zw_k$, i.e., if we change notation, into $Q^{k+1,r}\zw_k$, which generates new signs. Combining all these signs, we actually get $(-1)^r$. If $r=2$, the bracket $\ell''_2$ contains three terms. The proof is just a matter of computation. Since $\ell''_r,\zr''$ define a Lie $n$-antialgebroid structure on $E$, the same holds obviously true for $\ell'_r,\zr'$, so that, to complete the proof, it suffices to consider the associated Lie $n$-algebroid $(sE, (\ell_r)_r,\zr)$.\medskip

The constructions of a higher Lie algebroid from a higher Q-manifold and vice versa are of course inverses of each other. \end{proof}

\noindent{\bf Addendum.} The sign in the first term of (\ref{Signs1}) is $-\ze(\zs)$. If we denote, for simplicity, the degree of $X_j$ by $X_j$ instead of $-a_j$, those in the four terms of (\ref{Signs2}) are $$(-1)^{X_i(X_{i+1}+\ldots+X_r)+k(X_1+\ldots\hat{i}\ldots + X_r)}, (-1)^{k+X_i(X_1+\ldots+X_{i-1}+k+1)}\;,$$ $$(-1)^{k+1+X_i(X_1+\ldots+X_{i-1}+k+1)}, (-1)^{X_i(X_1+\ldots+X_{i-1}+k+1)}\;.$$ It is thus clear that the first term of (\ref{Signs1}) vanishes and that the second and third terms of (\ref{Signs2}) cancel. Moreover, (\ref{Signs3}) vanishes in view of (\ref{ReprCond1}), independently of the involved signs. When writing explicitly the terms of (\ref{Signs4}), we get for instance $$(-1)^{(X_1+k)(X_2+X_3)}\left[\zr'(\ell'_2(X_2,X_3))\zw_k(X_1)\right.$$ $$\left.+(-1)^{X_2}\zr'(X_2)\zr'(X_3)\zw_k(X_1)+(-1)^{(X_2+1)X_3}\zr'(X_3)\zr'(X_2)\zw_k(X_1)\right]\;.$$ It now suffices to observe that this sum vanishes, if $X_1\neq -1$ or $X_2\neq -1$, and that otherwise it reads $$\zr'(\ell'_2(X_2,X_3))\zw_k(X_1)-\zr'(X_2)\zr'(X_3)\zw_k(X_1)+\zr'(X_3)\zr'(X_2)\zw_k(X_1)\;$$ and thus vanishes in view of (\ref{ReprCond2}).\medskip

\noindent{\bf Remark 7.} The preceding proof shows two facts: \begin{itemize} \item For any split Lie $n$-algebroid $(L,(\ell_r)_r,\zr)$ over a manifold $M$, the bundle map $\zr:L_{0}\to TM$ verifies $$\zr(\ell_2(X,Y))=[\zr(X),\zr(Y)]\;,$$ for all $X,Y\in\zG(L_0)$, where $[-,-]$ is the bracket of vector fields. In other words, the $n$-algebroid anchor is a representation on $\op{Vect}(M)$ of the Lie algebra (up to homotopy) bracket $\ell_2$ on $\zG(L_0)$.

\item Any Lie $n$-algebroid $(L,(\ell_r)_r,\zr)$ is implemented by higher derived brackets. Indeed, let $Q$ be the homological vector field associated to the Lie $n$-antialgebroid structure $\ell'_r=(-1)^{r(r-1)/2}s^{-1}\ell_rs^r$, $\zr'=\zr s$. From $Q$ we construct via higher derived brackets the antialgebroid structure $\ell_r'',\zr''$, and we reconstruct $\ell'_r,\zr'$. Hence, $\ell_r=s\ell'_r(s^{-1})^r=(-1)^rs\ell''_r(s^{-1})^r$ and $\zr=\zr's^{-1}=\zr'' s^{-1}$.  \end{itemize}

\noindent{\bf Remark 8.} For $n=1$, i.e. in the Lie algebroid case, the Chevalley-Eilenberg differential (\ref{LieNAldCEOper}) coincides with the de Rham differential of the considered Lie algebroid. More precisely, the shifting operator allows to interpret the Chevalley-Eilenberg differential $Q\in\op{Diff}^1\zG(\odot E^*)$ of the Lie $n$-algebroid $(sE,(\ell_r)_r,\zr)$ as differential $\tilde Q$ on $\zG(\boxdot(sE)^*)$. The computation is technical and will not be given here. If $\zh_{k,s}\in \zG((sE)^*_{-a_1+1}\boxdot\ldots\boxdot(sE)^*_{-a_s+1})$, $\sum a_j=k$, we find \be\tilde Q^{k+1,r}\zh_{k,s}=(-1)^{(r-s+1)(s-1)}\zh_{k,s}\circ (\ell_{r-s+1}\boxdot\op{id}_{s-1})-\zr\boxdot\zh_{k,s}\;,\label{LieNAldCEOperShift}\ee where $\op{id}_{s-1}(X_1,\ldots,X_{s-1})=X_1\boxdot\ldots\boxdot X_{s-1}$. In the case $n=1$, necessarily $s=k, r=k+1$, and $\zh_{k,s}=:\zh_{k}\in\zG(\w^k(sE)_0^*)$. It is easily seen that Equation (\ref{LieNAldCEOperShift}) then reduces to the usual de Rham cohomology operator.

\section{Geometry of Lie $n$-algebroid morphisms}

\subsection{General morphisms of Lie $n$-algebroids}

In this section, we define morphisms between Lie $n$-algebroids over different bases in terms of anchors and brackets. In the case $n=1$, we recover the notion of Lie algebroid morphism \cite{Mac05}, and for $n$-algebroids over a point, the new concept coincides with that of Lie infinity algebra morphism.\medskip

Let $E=\oplus_{i=1}^nE_{-i}$ (resp., $F=\oplus_{i=1}^nF_{-i}$) be a graded vector bundle over $M$ (resp., $N$). A graded vector bundle morphism (in the categorical sense, i.e. a vector bundle morphism of degree 0) $\zf'_r:\odot^rE\to F$, $r\ge 1$, is a smooth map over a smooth map $\zf_0:M\to N$, with linear restrictions to the fibers. For instance, $$\zf'_2:\w^2E_{-1,x}\to F_{-2,\zf_0(x)},\;\zf'_2:E_{-1,x}\0 E_{-2,x}\to F_{-3,\zf_0(x)}, ...$$are linear. Remark that if $r\ge n+1$, the highest degree in $\odot^rE$ is $-r\le -n-1<-n$, so that $\zf'_r$ is necessarily zero. A graded vector bundle morphism $\zf'_r:\odot^rE\to F$ can be viewed as a vector bundle morphism $\zf_r:\boxdot^rsE\to sF$ of degree $1-r$: $$\zf_r=s\,\zf'_r(s^{-1})^r\quad\text{and}\quad\zf'_r=(-1)^{r(r-1)/2}s^{-1}\zf_r\, s^r\;.$$

If $X_i\in\zG(E_{-a_i})$, $i\in\{1,\ldots,r\}$, then $\zf'_r\circ X:=\zf'_r\circ(X_1\odot\ldots\odot X_r)$ has obviously a decomposition of the form \be\label{PhiDecom}\zf'_r\circ X=\sum_jf_j^{X}\;\xi_{j}^{X}\circ\zf_0\;,\ee where the sum is finite, $f_j^X\in\Ci(M)$ and $\xi_j^X\in\zG(F_{-\sum a_i}).$ Indeed, it suffices to take as $\xi_j^X$ a finite generating family of sections in the $\Ci(N)$-module $\zG(F_{-\sum a_i})$. Furthermore, it is easily seen that the graded symmetric tensor product $\zf'_{t_1}\odot\ldots\odot\zf'_{t_r}$ of graded vector bundle morphisms is given as follows. If $X_i\in\zG(E_{-a_i})$, $i\in\{1,\ldots,t\}$, and $t_1+\ldots+t_r=t$, $t_j\neq 0$, then \be\label{PhiDecomp}(\zf'_{t_1}\odot\ldots\odot\zf'_{t_r})\circ(X_1,\ldots,X_t)=\sum_{\zs\in\op{Sh}(t_1,\ldots,t_r)}\sum_{j_1}\ldots\sum_{j_r}\ze(\zs)\;f_{j_1}^{X^{\zs^1}}\ldots f_{j_r}^{X^{\zs^r}}(\xi_{j_1}^{X^{\zs^1}}\odot\ldots\odot\xi_{j_r}^{X^{\zs^r}})\circ\zf_0\;,\ee where $\ze(\zs)$ is the Koszul sign.\medskip

Let $\ell_i,\zr$ (resp., $m_i,r$) be a Lie $n$-algebroid structure on $sE$ (resp., $sF$), and denote by $\ell'_i,\zr'$ (resp., $m'_i,r'$) the corresponding Lie $n$-antialgebroid structure on $E$ (resp., $F$).

\begin{defi}\label{LadMorDef} A \emph{morphism of Lie $n$-algebroids} between $sE$ and $sF$ is a family $\zf_r:\boxdot^rsE\to sF$, $1\le r\le n$, of degree $1-r$ vector bundle morphisms over a base map $\zf_0:M\to N$, such that \be\label{LadMorCond1}r'\circ\zf'_1=T\zf_{0}\circ\zr'\;,\ee as well as, for any $1\le t\le n+1$ and any homogeneous sections $X_i$ of $E$, $i\in\{1,\ldots,t\}$, with decompositions $$\zf'_r\circ X_I=\sum_jf_j^{X_I}\;\xi_{j}^{X_I}\circ\zf_0\;$$ (for any $r$ and any product $X_I:=X_{i_1}\odot\ldots\odot X_{i_r}$),  $$\sum_{r+s=t+1}\sum_{\zs\in\op{Sh}(s,r-1)}\ze(\zs)\;\zf'_r\circ(\ell'_s(X_{\zs_1},\ldots,X_{\zs_s})\odot X_{\zs_{s+1}}\odot\ldots\odot X_{\zs_t})$$ $$+\sum_{ij}(-1)^{\tilde X_i(\tilde X_1+\ldots +\tilde X_{i-1})+1}\;(\zr' (X_i)f_j^{X_1\ldots\hat\imath\ldots X_t})\xi_j^{X_1\ldots\hat\imath\ldots X_t}\circ\zf_0$$ \be\label{LadMorCond2}=\sum_{r=1}^t\frac{1}{r!}\sum_{\tiny\begin{array}{c}t_1+\ldots+t_r=t\\t_j\neq 0\end{array}}\sum_{\zs\in\op{Sh}(t_1,\ldots,t_r)}\sum_{j_1}\ldots\sum_{j_r}\ze(\zs)f_{j_1}^{X^{\zs^1}}\ldots f_{j_r}^{X^{\zs^r}}m'_r(\xi_{j_1}^{X^{\zs^1}},\ldots,\xi_{j_r}^{X^{\zs^r}})\circ\zf_0\;,\ee where $\tilde X_k$ denotes the degree of $X_k$.\end{defi}

Note that for $t\ge n+2$, the highest degree of the terms in Equation (\ref{LadMorCond2}) is $1-t\le -n-1<-n$, so that any term necessarily vanishes.\medskip

\noindent{\bf Remark 9.}\begin{itemize} \item This definition is the geometric translation of the natural supergeometric / algebraic definition of Lie $n$-algebroid morphisms, see below.

\item For $n=1$, the definition reduces to that of morphisms of Lie algebroids over different bases, see \cite{HM90}, \cite{BKS}, \cite{Mac05}.

Indeed, note first that for $n=1$, the maps $\zf'_r$, $r\neq 1$, vanish, as they are of degree 0. We already noticed that the same is true for $\ell'_r,m'_r$, $r\neq 2$.

For $t\neq 2$, Condition (\ref{LadMorCond2}) is trivial. To understand this claim, observe that the sum in the second row of (\ref{LadMorCond2}) (resp., the {\small RHS} of (\ref{LadMorCond2})) is constructed from the decomposition (\ref{PhiDecom}) of $\zf'_{t-1}\circ(X_1\odot\ldots\hat\imath\ldots\odot X_t)$ (resp., the decomposition (\ref{PhiDecomp}) of $(\zf'_{t_1}\odot\ldots\odot\zf'_{t_r})\circ(X_1,\ldots,X_t)$). It is now clear that the general term of the sum in the first row of (\ref{LadMorCond2}) is nonzero only if $r=1$ and $s=2$, hence if $t=2$; that the sum in the second row does not vanish only if $t=2$; that the {\small RHS} does not vanish only if $r=2$ and $t_1=t_2=1$, hence, if $t=2$.

Eventually, for $t=2$, Equation (\ref{LadMorCond2}) is easily written in terms of $\zf_1,\ell_2,\zr,m_2$. It then coincides with the similar condition in the aforementioned works.

\item A priori Definition \ref{LadMorDef} depends on the choice of the involved decompositions. However, it is known, at least in the Lie algebroid case $n=1$, that all the terms are well-defined, see \cite{BKS}, \cite{Mac05}. For $n> 1$, this fact is a consequence of Theorem \ref{LadMorTheo}, see below.\end{itemize}

Before continuing, we work out an equivalent version of the anchor condition (\ref{LadMorCond1}), which uses the decomposition (\ref{PhiDecom}). Let $g\in \Ci(N)$, let $X\in\zG(E_{-1})$, and let all the other objects be as above. Remember first that, if $Z_x\in T_xM$, $x\in M,$ we have $Z_x(g\circ \zf_0)=(d_{\zf_0(x)}g)((T_x\zf_0)Z_x)=((T_x\zf_0)Z_x)(g)$, and that, if $Y\in\op{Vect}(N)$, we get $Y_{\zf_0(x)}g=(Yg)(\zf_0(x))=(\zf_0^*(Yg))(x).$ Assume now that \be\label{PhiDecomp'}\zf'_1\circ X=\sum_jf^X_j\xi^X_j\circ\zf_0\;.\ee When using the just recalled results and taking into account the decomposition (\ref{PhiDecomp'}), we see that Equation (\ref{LadMorCond1}) is equivalent to the equation $$\lp\zr'(X)(\zf_0^*g)\rp_x= \zr'(X_x)(g\circ\zf_0)=((T_x\zf_0)(\zr'(X_x)))(g)=r'(\zf'_1X_x)(g)$$ \be\label{AnchorDecomp}=\sum_jf^X_j(x)r'(\xi^X_{j})_{\zf_0(x)}(g)=\lp\sum_jf_j^X\zf^*_0(r'(\xi_j^X)g)\rp(x)\;.\ee

\subsection{Base-preserving morphisms of Lie $n$-algebroids}

If $\zf_0:M\to N$ is a diffeomorphism, the Lie $n$-algebroid morphism conditions can be simplified. Indeed, identify the manifolds $M$ and $N$, so that $\zf_0=\op{id}$.

The anchor condition (\ref{LadMorCond1}) then reduces to \be r'\circ\zf'_1=\zr'\;,\label{LadMorCondSimpl1}\ee which is equivalent to $r'(\zf'_1\circ X)=\zr'(X)$, for all $X\in\zG(E)$, provided we define $\zr'$ and $r'$ by 0 in all degrees different from $-1$.

As for the condition (\ref{LadMorCond2}), let us work -- to simplify -- up to sign. Remember first that $m'_r$, $r\neq 2$, is $\Ci(N)$-multilinear and that $m'_2$ verifies, for any $f,g\in\Ci(N)$ and any $X,Y\in\zG(F)$, $$fg\,m'_2(X,Y)=m'_2(fX,gY)+f(r'(X)g) Y+g(r'(Y)f)X,$$where we use again the just mentioned extension of $r'$ by 0. The anchor terms in the {\small LHS} of (\ref{LadMorCond2}) then read \be\label{MorSimpl1}\sum_{i,\ell}(r'(\zf'_1\circ X_i)f_{\ell}^{X_{\hat\imath}})\xi_{\ell}^{X_{\hat\imath}}\;,\ee where $X_{\hat\imath}=X_1\ldots\hat\imath\ldots X_t$. In view of Equation (\ref{PhiDecomp}), the {\small RHS} of (\ref{LadMorCond2}) is given by \be\label{MorSimpl2}\sum_{r=1}^t\frac{1}{r!}\sum_{\tiny\begin{array}{c}t_1+\ldots+t_r=t\\t_j\neq 0\end{array}}m'_r\lp(\zf'_{t_1}\odot\ldots\odot\zf'_{t_r})\circ(X_1,\ldots,X_t)\rp+\ldots\;,\ee where $\ldots\,$ denote the anchor terms that appear if $r=2$.

If $t=1$, there are no such terms; on the other hand, the sum (\ref{MorSimpl1}) then vanishes (we will refer to this observation as result ($\star$)). Assume in the following that $t\ge 2$. The potential anchor terms are generated by the transformation of the sum $$\frac{1}{2}\sum_{\tiny\begin{array}{c}t_1+t_2=t\\ t_i\neq 0\end{array}}\sum_{\zs\in\op{Sh}(t_1,t_2)}\sum_{j_1}\sum_{j_2}f_{j_1}^{X^{\zs^1}}f_{j_2}^{X^{\zs^2}}m'_2(\xi_{j_1}^{X^{\zs^1}},\xi_{j_2}^{X^{\zs^2}})\;.$$If the total degree of $X_{\zs_1},\ldots,X_{\zs_{t_1}}$ and the total degree of $X_{\zs_{t_1+1}},\ldots,X_{\zs_{t_1+t_2}}$ differ both from $-1$, no anchor terms appear. Otherwise, $t_1=1$ (and $t_2=t-1$) or $t_2=1$ (and $t_1=t-1$). These possibilities correspond to different terms in the sum over $t_1,t_2$ if and only if $t\ge 3.$

Let now $t\ge 3$. In view of what has been said, additional terms appear only in the two mentioned cases. They are given by $$\frac{1}{2}2\sum_{i,j,\ell}\lp f_j^{X_i}(r'(\xi_j^{X_i})f_{\ell}^{X_{\hat\imath}})\xi_{\ell}^{X_{\hat\imath}}+f_{\ell}^{X_{\hat\imath}}(r'(\xi_{\ell}^{X_{\hat\imath}})f_j^{X_i})\xi_j^{X_i}\rp=\sum_{i,j,\ell} f_j^{X_i}(r'(\xi_j^{X_i})f_{\ell}^{X_{\hat\imath}})\xi_{\ell}^{X_{\hat\imath}}$$ \be\label{MorSimpl3}=\sum_{i,\ell} (r'(\zf'_1\circ X_i)f_{\ell}^{X_{\hat\imath}})\xi_{\ell}^{X_{\hat\imath}}\;,\ee since the degree of $\xi_{\ell}^{X_{\hat\imath}}$ is $<-1$.

If $t=2$, the sum over $t_1,t_2$ contains a unique term $t_1=t_2=1$ and the anchor terms (although possibly zero) read \be\label{MorSimpl4}\frac{1}{2}2\sum_{j,\ell}\lp f_j^{X_1}(r'(\xi_j^{X_1})f_{\ell}^{X_2})\xi_{\ell}^{X_2}+
f_j^{X_2}(r'(\xi_j^{X_2})f_{\ell}^{X_1})\xi_{\ell}^{X_1}\rp$$ $$=\sum_{\ell}\lp (r'(\zf'_1\circ X_1)f_{\ell}^{X_2})\xi_{\ell}^{X_2}+(r'(\zf'_1\circ X_2)f_{\ell}^{X_1})\xi_{\ell}^{X_1}\rp$$ $$=\sum_{i,\ell} (r'(\zf'_1\circ X_i)f_{\ell}^{X_{\hat\imath}})\xi_{\ell}^{X_{\hat\imath}}\;.\ee

Since the sums (\ref{MorSimpl1}) and (\ref{MorSimpl3}) or (\ref{MorSimpl4}) cancel out (see also ($\star$)), the simplified form of the algebroid morphism condition (\ref{LadMorCond2}) follows. Hence, the next reformulation.

\begin{defi}\label{BasePresLieMorphDef} Let $sE$ and $sF$ be two Lie $n$-algebroids over a same base. A \emph{base-preserving morphism of Lie $n$-algebroids} between $sE$ and $sF$ is a family $\zf_r:\boxdot^rsE\to sF$, $1\le r\le n$, of degree $1-r$ vector bundle morphisms (over the identity) that verify the condition \be\label{LadMorCondSimplified1}r'\circ\zf'_1=\zr'\;,\ee as well as, for any $1\le t\le n+1$ and any homogeneous sections $X_i$ of $E$, $i\in\{1,\ldots,t\}$, the condition $$\sum_{r+s=t+1}\sum_{\zs\in\op{Sh}(s,r-1)}\ze(\zs)\zf'_r\circ\lp\ell'_s(X_{\zs_1},\ldots,X_{\zs_s})\odot X_{\zs_{s+1}}\odot\ldots\odot X_{\zs_t}\rp$$ \be\label{LadMorCondSimplified2}=\sum_{r=1}^t\frac{1}{r!}\sum_{\tiny\begin{array}{c}t_1+\ldots+t_r=t\\t_j\neq 0\end{array}} \;m'_r\lp(\zf'_{t_1}\odot\ldots\odot\zf'_{t_r})\circ(X_1,\ldots,X_t)\rp\;.\ee \end{defi}

\noindent{\bf Remark 10.} Remember that a Lie $n$-algebroid over a point is exactly a Lie $n$-algebra, hence a truncated Lie infinity algebra.

When rewriting the condition (\ref{LadMorCondSimplified2}) in terms of $\zf_r,\ell_r,$ and $m_r$, we obtain $$\sum_{r+s=t+1}\sum_{\zs\in\op{Sh}(s,r-1)}(-1)^{s(r-1)}\op{sign}\zs\,\ze(\zs)\zf_r\circ \lp\ell_s(Y_{\zs_1},\ldots,Y_{\zs_s}),Y_{\zs_{s+1}},\ldots,Y_{\zs_t}\rp$$ \be\sum_{r=1}^t\frac{1}{r!}\sum_{\tiny\begin{array}{c}t_1+\ldots+t_r=t\\t_j\neq 0\end{array}}\sum_{\zs\in\op{Sh}(t_1,\ldots,t_r)}\pm \op{sign}\zs\,\ze(\zs)\,m_r(\zf_{t_1}\circ Y^{\zs^1},\ldots,\zf_{t_r}\circ Y^{\zs^r})\;,\label{LieInftyAlgMorph}\ee where we wrote $Y_i$ instead of $sX_i$ and where $$\pm =(-1)^{r(r-1)/2+\sum_jt_j(r-j)}+\sum_j|Y^{\zs^j}|(r-j+t_{j+1}+\ldots+t_r)\;,$$ $|Y^{\zs^j}|$ being the sum of the degrees of the components of $Y^{\zs^j}$. This is exactly the Lie infinity algebra morphism condition, see \cite{Sch04}, \cite{AP10}, \cite{LV11}. Hence, the definition of base-preserving Lie $n$-algebroid morphisms coincides over a point (bundles become spaces, bundle morphisms become linear maps, anchors vanish, sections become vectors and compositions evaluations) with the definition of (truncated) Lie infinity algebra morphisms. We thus prove a result conjectured in \cite{SZ11}, Remark 2.5.\medskip

\subsection{Categories of Lie $n$-algebroids and NQ-manifolds: comparison of morphisms}

In this section we show that morphisms of split Lie $n$-algebroids are morphisms of NQ-manifolds between split NQ-manifolds.

\begin{prop}\label{NManMorph} There is a 1-to-1 correspondence between families $\zf_r:\boxdot^rsE\to sF$, $1\le r\le n$, of degree $1-r$ vector bundle morphisms over a map $\zf_0$, and graded algebra morphisms $\Phi:\zG(\odot F^*)\to \zG(\odot E^*)$.\end{prop}

\newcommand{\ol}{\odot\ldots\odot}

\begin{proof} To define $\Phi:\zG(\odot F^*)\to \zG(\odot E^*)$, we define, for $\zh_{k,s}\in\zG(F^*_{-b_1}\ol F^*_{-b_s})$, $\sum b_j=k$, the projection of $\Phi^{k,r}\zh_{k,s}$ onto any $\zG(E^*_{-a_1}\ol E^*_{-a_r})$, $\sum a_j=k$. More precisely, we define $(\Phi^{k,r}\zh_{k,s})(X_1,\ldots,X_r)\in\Ci(M)$, $X_j\in \zG(E_{-a_j})$, in a way such that the dependence on the $X_j$ be $\Ci(M)$-multilinear and graded symmetric.

We first set \be\label{DefPhi0}\Phi^{0,0}:\Ci(N)\ni g\mapsto g\circ\zf_0\in\Ci(M)\;.\ee Then, for $k\ge 1$, we define $(\Phi^{k,r}\zh_{k,s})(X_1,\ldots,X_r)$ by $0$, if $s> r$, and set, for $s\le r$, \be\label{DefPhi}(\Phi^{k,r}\zh_{k,s})(X_1,\ldots,X_r)=\la\zh_{k,s}\circ\zf_0,\frac{1}{s!}\sum_{\tiny\begin{array}{c}r_1+\ldots+r_s=r\\r_i\neq 0\end{array}}(\zf'_{r_1}\ol\zf'_{r_s})\circ(X_1,\ldots,X_r)\ra\;.\ee Indeed, for any $x\in M$, we have $$\frac{1}{s!}\sum_{\tiny\begin{array}{c}r_1+\ldots+r_s=r\\r_i\neq 0\end{array}}(\zf'_{r_1}\ol\zf'_{r_s})(X_{1,x},\ldots,X_{r,x})$$ \be =\frac{1}{s!}\sum_{\tiny\begin{array}{c}r_1+\ldots+r_s=r\\r_i\neq 0\end{array}}\sum_{\zs\in\op{Sh}(r_1,\ldots,r_s)}\ze(\zs)\;\zf'_{r_1}(X^{\zs^1}_x)\ol\zf'_{r_s}(X^{\zs^s}_x)\;,\label{Argument}\ee where a notation as $X^{\zs^1}_x$ means $X_{\zs_1,x},\ldots,X_{\zs_{r_1},x}$. If we denote the sum of the degrees $-a_{\zs_j}$ of these $X_{\zs_j,x}$ by $|X^{\zs^1}_x|$, we get $$\zf'_{r_1}(X^{\zs^1}_x)\ol\zf'_{r_s}(X^{\zs^s}_x)\in  F_{|X^{\zs^1}_x|,\zf_0(x)}\ol F_{|X^{\zs^s}_x|,\zf_0(x)}\;.$$ On the other hand, $\zh_{k,s;\zf_0(x)}$ is an element of $(F_{-b_1;\zf_0(x)}\ol F_{-b_s;\zf_0(x)})^*$. Of course, the contraction of the terms of the {\small RHS} of (\ref{Argument}) with $\zh_{k,s;\zf_0(x)}$ gives a nonzero contribution only if the considered term belongs to the source space of $\zh_{k,s;\zf_0(x)}$. It is now clear that the {\small RHS} of (\ref{DefPhi}) is a function on $M$ that depends on the $X_j$ in a $\Ci(M)$-multilinear and graded symmetric way.

The definition of $\Phi:\zG(\odot F^*)\to \zG(\odot E^*)$ is now complete. In view of (\ref{CohomoWedge}) and (\ref{DefPhi0}), (\ref{DefPhi}), $\Phi$ is a graded algebra ({\small GA}) morphism.\medskip

\noindent{\bf Remark 11.} It is easily checked that, for $n=1$, $E=TM$, $F=TN$ and $\zf_1=T\zf_0$, the algebra morphism $\Phi$ is just the pullback $\zf_0^*:\zG(\w T^*N)\to \zG(\w T^*M)$ of differential forms by $\zf_0$.\medskip

{\it Proof (continuation).} Conversely, to any {\small GA} morphism $\Phi:\zG(\odot F^*)\to \zG(\odot E^*)$ we can associate a family $\zf'_r:\odot^rE\to F$, $r\ge 1$, of graded vector bundle morphisms over a map $\zf_0$.

The map $\Phi$ is in particular an associative algebra morphism $\Phi:\Ci(N)\to\Ci(M)$. Hence, it is the pullback by a smooth map $\zf_0:M\to N$, see e.g. \cite{AMR83}, \cite{Bkouche65}. It follows that, for any $g\in\Ci(N)$ and $\zh\in\zG(\odot F^*)$, $$\Phi(g\zh)=(g\circ\zf_0)(\Phi \zh)\;.$$ This `function-linearity' implies as usual that $\Phi$ is local, i.e. that $\Phi \zh=0$ on $\zf_0^{-1}(V)$, if $\zh=0$ on $V$, where $V$ is an open subset of $N$ (indeed, for any $x\in\zf_0^{-1}(V)$, consider a bump function $\za$ around $\zf_0(x)$, and note that $\Phi \zh=\Phi((1-\za)\,\zh)$). In fact, for any $x\in M$, we even have $(\Phi \zh)_x=0$, if $\zh_{\zf_0(x)}=0$ (indeed, take a local frame $(b_i)_i$ of $\odot F^*$ in $V\ni\zf_0(x)$ and set $\zh=\sum_is^ib_i$ in $V$; if the bump function $\za$ is as above and has value 1 in $W\ni\zf_0(x)$, then $\zh=\sum_i(\za s^i)(\za b_i)$ in $W$; due to locality, $(\Phi \zh)_x=\sum_i(\za s^i)_{\zf_0(x)}(\Phi(\za b_i))_x=0$): the value $(\Phi\zh)_x$, $x\in M,$ only depends on the value $\zh_{\zf_0(x)}$.

To define, for $r\ge 1$ and $x\in M,$ a linear map $\zf'_r:\odot^rE_x\to F_{\zf_0(x)}$ of degree 0, associate to any $p\in E_{-a_1,x}\odot\ldots\odot E_{-a_r,x}\subset \odot^rE_x$, $\sum a_j=k,$ a unique $$\zf'_r(p)\in F_{-k,\zf_0(x)}\simeq (F_{-k,\zf_0(x)}^*)^*\;.$$ Hence, let $q^*\in F_{-k,\zf_0(x)}^*$ and choose $\zh\in \zG(F^*_{-k})$ such that $\zh_{\zf_0(x)}=q^*.$ The value $(\Phi\zh)_x$ is well-defined in $\odot E^*_x$ and has degree $k$. It suffices now to set \be\label{DefZf}\zf'_r(p)(q^*)=\la p,(\Phi\zh)_x\ra\in\R\;,\ee where of course only the projection of $(\Phi\zh)_x$ onto $E^*_{-a_1,x}\odot\ldots\odot E^*_{-a_r,x}$ gives a nonzero contribution.\medskip

The definitions (\ref{DefPhi}) and (\ref{DefZf}) are in fact inverses of each other. Indeed, if $p=X_{1,x}\odot\ldots\odot X_{r,x}$, $X_{j,x}\in E_{-a_j,x}$, choose $X_j\in\zG(E_{-a_j})$ (resp., $\zh_{k,1}\in\zG(F^*_{-k})$) that extends $X_{j,x}$ (resp., $q^*$). Definition (\ref{DefZf}) then reads $$\zf'_r(X_{1,x}\odot\ldots\odot X_{r,x})(\zh_{k,1;\zf_0(x)})=\la X_{1,x}\odot\ldots\odot X_{r,x},(\Phi^{k,r}\zh_{k,1})_x\ra\;.$$ \end{proof}

\begin{cor} There is a 1:1 correspondence between graded vector bundle morphisms $\zf':E\to F$ and bigraded algebra morphisms $\Phi:\zG(\odot F^*)\to \zG(\odot E^*)$, i.e. algebra morphisms that respect the standard and the homological degrees.\end{cor}

\noindent{\bf Remark 12.} We thus recover the result that the morphisms of split N-manifolds are the morphisms of graded vector bundles. Let us stress that the morphisms $\Phi:\zG(\odot F^*)\to \zG(\odot E^*)$ of graded algebras we considered in Proposition \ref{NManMorph}, are the morphisms of N-manifolds between the split N-manifolds $E[\cdot]$ and $F[\cdot]$ (split N-manifolds are not a full subcategory of N-manifolds).

\begin{proof} The corollary is a direct consequence of the proof of the preceding proposition. Indeed, if $\zf'_1$ is the unique map of the family of morphisms, it follows from Definition (\ref{DefPhi}) that $\Phi$ respects both degrees. Conversely, if $\Phi$ is a bigraded algebra morphism, Equation (\ref{DefZf}) provides only a map $\zf'_1$.\end{proof}

The next theorem explains our definition of Lie $n$-algebroid morphisms.

\begin{theo}\label{LadMorTheo} There is a 1-to-1 correspondence between morphisms of split Lie $n$-algebroids from $sE$ to $sF$ and morphisms of differential graded algebras from $(\zG(\odot F^*),Q_F)$ to $(\zG(\odot E^*),Q_E)$.\end{theo}

\noindent{\bf Remark 13.} This theorem means that the morphisms between the split Lie $n$-algebroids $(sE,\ell_i,\zr)$ and $(sF,m_i,r)$ are the morphisms of NQ-manifolds between the split NQ-manifolds $(E[\cdot],Q_E)$ and $(F[\cdot],Q_F)$. The point is that morphisms of split Lie $n$-algebroids are not necessarily morphisms of graded vector bundles. This observation is not surprising: morphisms of Lie infinity algebras are on their part usually not morphisms of graded vector spaces.\medskip

Let us first note that, in view of Proposition \ref{NManMorph}, Theorem \ref{LadMorTheo} just means that the morphism conditions (\ref{LadMorCond1}) and (\ref{LadMorCond2}) are equivalent to the equivariance condition \be\label{EquivCond}Q_E\circ\Phi=\Phi\circ Q_F\;\ee -- which proves that Definition \ref{LadMorDef} is independent of the chosen decompositions. More precisely, the equivariance condition is satisfied on the whole algebra $\zG(\odot F^*)$ if and only if it is satisfied on the generators $g\in\Ci(N)$ and $\zh_{k,1}\in\zG(F^*_{-k})$, $k\in\{1,\ldots,n\}$. It will turn out that the condition (\ref{EquivCond}) written on functions is equivalent to the condition (\ref{LadMorCond1}), and that (\ref{EquivCond}) written on generators of degrees $k\in\{1,\ldots,n\}$ is equivalent to the conditions (\ref{LadMorCond2}).\medskip

\begin{proof} To simplify, we work in this proof up to sign. However, some signs are needed to explain Definition \ref{LadMorDef}. We denote them by $(\pm_1) - (\pm_3)$ and write them explicitly at the end of the proof.\medskip

Let $g\in\Ci(N)$ and let $X\in\zG(E_{-1})$. We get $$(Q_E^{1,1}\Phi^{0,0} g)(X)=\zr'(X)(\zf_0^*g)\;$$and $$(\Phi^{1,1} Q_F^{1,1}g)(X)=\la (Q_F^{1,1}g)\circ\zf_0,\zf'_1\circ X\ra=\sum_jf^X_j\la Q_F^{1,1}g,\xi_j^X\ra\circ\zf_0=\sum_jf^X_j\zf_0^*(r'(\xi_j^X)g)\;.$$In view of Equation (\ref{AnchorDecomp}), this means that $Q_E\circ\Phi$ and $\Phi\circ Q_F$ coincide on functions if and only if Equation (\ref{LadMorCond1}) holds true.

Let now $k\in\{1,\ldots,n\}$, $\zh_{k,1}\in\zG(F^*_{-k})$, and  $t\in\{1,\ldots,k+1\}$. We will compute \be\label{LadMorC20}\sum_{r=1}^{k+1}Q_E^{k+1,t}\Phi^{k,r}\zh_{k,1}-\sum_{r=1}^{k+1}\Phi^{k+1,t} Q_F^{k+1,r}\zh_{k,1}\in\,^t\!\cA^{k+1}\;\ee on $(X_1,\ldots, X_t)$, $X_j\in\zG(E_{-a_j})$, $\sum a_j=k+1$.

When applying the definitions of $Q_E$ and $\Phi$, we get $$\sum_{r=1}^{k+1}(Q_E^{k+1,t}\Phi^{k,r}\zh_{k,1})(X_1,\ldots,X_t)$$ $$=\sum_{r=1}^t \sum_{\zs\in\op{Sh}(t-r+1,r-1)}(\Phi^{k,r}\zh_{k,1})(\ell'_{t-r+1}(X_{\zs_1},\ldots,X_{\zs_{t-r+1}})\odot X_{\zs_{t-r+2}}\odot\ldots\odot X_{\zs_t})$$ $$+\sum_{i}\zr'(X_i) (\Phi^{k,t-1}\zh_{k,1})(X_1,\ldots\hat\imath\ldots,X_t)$$ \be\label{LadMorC21}=\la \zh_{k,1}\circ\zf_0,\sum_{r+s=t+1}\sum_{\zs\in\op{Sh}(s,r-1)}(\pm_1)\,\zf'_r\circ(\ell'_{s}(X_{\zs_1},\ldots,X_{\zs_{s}})\odot X_{\zs_{s+1}}\odot\ldots\odot X_{\zs_t})\ra\ee $$+\sum_{i,j}\zr'(X_i)\lp f_j^{X_{\hat \imath}}\la\zh_{k,1}\circ\zf_0,\xi_j^{X_{\hat \imath}}\circ\zf_0\ra\rp\;,$$where $X_{\hat\imath}$ stands for $(X_1,\ldots\hat\imath\ldots,X_t)$. The last sum reads \be\label{LadMorC22}\la\zh_{k,1}\circ\zf_0,\sum_{i,j}(\pm_2)\,(\zr'(X_i) f_j^{X_{\hat \imath}})\xi_j^{X_{\hat \imath}}\circ\zf_0\ra\ee $$+\sum_{i,j} f_j^{X_{\hat \imath}}\zr'(X_i)\zf_0^*\la\zh_{k,1},\xi_j^{X_{\hat \imath}}\ra\;.$$Observe that, independently of the implication we have in mind, (\ref{LadMorCond1}) and (\ref{LadMorCond2}) imply (\ref{EquivCond}) or (\ref{EquivCond}) implies (\ref{LadMorCond1}) and (\ref{LadMorCond2}), we can assume at this stage that (\ref{LadMorCond1}) and its equivalent form (\ref{AnchorDecomp}) hold true. It follows that the last sum of the preceding expression can be written in the form \be\label{LadMorC23}\sum_{i,j,\ell} f_{\ell}^{X_i}f_j^{X_{\hat \imath}}\zf_0^*\lp r'(\xi_{\ell}^{X_i})\la\zh_{k,1},\xi_j^{X_{\hat \imath}}\ra\rp\;.\ee The reader has probably noticed that many of the terms we write and transform are zero. The point is that it is much easier to transform sums with potentially vanishing terms, than to work with the actually present terms.

On the other hand, when using Equation (\ref{PhiDecomp}), we get $$\sum_{r=1}^{k+1}(\Phi^{k+1,t} Q_F^{k+1,r}\zh_{k,1})(X_1,\ldots,X_t)$$ $$=\sum_{r=1}^t\frac{1}{r!}\sum_{\tiny\begin{array}{c}\tiny t_1+\ldots+t_r=t\\ t_i\neq 0\end{array}}\sum_{\zs\in\op{Sh}(t_1,\ldots,t_r)}\sum_{j_1}\ldots\sum_{j_r}f_{j_1}^{X^{\zs^1}}\ldots f_{j_r}^{X^{\zs^r}}\zf_0^*\la Q_F^{k+1,r}\zh_{k,1}, \xi_{j_1}^{X^{\zs^1}}\odot\ldots\odot\xi_{j_r}^{X^{\zs^r}}\ra$$ \be\label{LadMorC24}=\la\zh_{k,1}\circ\zf_0,\sum_{r=1}^t\frac{1}{r!}\sum_{\tiny\begin{array}{c}\tiny t_1+\ldots+t_r=t\\ t_i\neq 0\end{array}}\sum_{\zs\in\op{Sh}(t_1,\ldots,t_r)}\sum_{j_1}\ldots\sum_{j_r}(\pm_3)\,f_{j_1}^{X^{\zs^1}}\ldots f_{j_r}^{X^{\zs^r}}m'_r(\xi_{j_1}^{X^{\zs^1}},\ldots,\xi_{j_r}^{X^{\zs^r}})\circ\zf_0\ra\ee $$+\frac{1}{2}\sum_{\tiny\begin{array}{c}\tiny t_1+t_2=t\\ t_i\neq 0\end{array}}\sum_{\zs\in\op{Sh}(t_1,t_2)}\sum_{j_1}\sum_{j_2}f_{j_1}^{X^{\zs^1}}f_{j_2}^{X^{\zs^2}}\zf_0^*\lp(r'\odot
\zh_{k,1})(\xi_{j_1}^{X^{\zs^1}},\xi_{j_2}^{X^{\zs^2}})\rp\;.$$We now examine the nonzero terms in the sum over $t_1,t_2$. Note first that if $t=1$, the entire sum vanishes. We thus can assume that $t\ge 2.$ The function that we pull back by $\zf_0$ is given by $$r'(\xi_{j_1}^{X^{\zs^1}})\la\zh_{k,1},\xi_{j_2}^{X^{\zs^2}}\ra+r'(\xi_{j_2}^{X^{\zs^2}})\la\zh_{k,1},\xi_{j_1}^{X^{\zs^1}}\ra$$and does therefore not vanish only if the sum of the degrees of $X_{\zs_1},\ldots,X_{\zs_{t_1}}$ or $X_{\zs_{t_1+1}},\ldots,X_{\zs_{t_1+t_2}}$ is $-1$. In this case, $t_1=1$ (and $t_2=t-1$) or $t_2=1$ (and $t_1=t-1$). As already observed above, the latter two possibilities correspond to different terms in the sum over $t_1, t_2$, only if $t\neq 2$. Moreover, since the sum of the degrees of $X_1,\ldots,X_t$ is $-k-1$, see above, we find that $t=2$, if $k=1$.

Consider first the case $t\neq 2$ (then $k\neq 1$). The sum over $t_1, t_2$ now reads $$\frac{1}{2}2\sum_i\sum_{j_1}\sum_{j_2}f_{j_1}^{X_i}f_{j_2}^{X_{\hat\imath}}\zf_0^*\lp r'(\xi_{j_1}^{X_i})\la\zh_{k,1},\xi_{j_2}^{X_{\hat\imath}}\ra+r'(\xi_{j_2}^{X_{\hat\imath}})\la\zh_{k,1},\xi_{j_1}^{X_i}\ra\rp$$ \be\label{LadMorC25}=\sum_{i,j,\ell}f_{\ell}^{X_i}f_{j}^{X_{\hat\imath}}\zf_0^*\lp r'(\xi_{\ell}^{X_i})\la\zh_{k,1},\xi_{j}^{X_{\hat\imath}}\ra\rp\;,\ee since the degree of $\xi_{j_2}^{X_{\hat\imath}}$, i.e. the degree of $X_{\hat\imath}$, is less than $-1$.

In case $t=2$, the sum over $t_1,t_2$ reads $$\frac{1}{2}2 \sum_{j,\ell}f_{\ell}^{X_1}f_j^{X_2}\zf_0^*\lp r'(\xi_{\ell}^{X_1})\la \zh_{k,1},\xi_j^{X_2}\ra +r'(\xi_j^{X_2})\la \zh_{k,1},\xi_{\ell}^{X_1}\ra\rp$$ \be\label{LadMorC26}=\sum_{i,j,\ell}f_{\ell}^{X_i}f_{j}^{X_{\hat\imath}}\zf_0^*\lp r'(\xi_{\ell}^{X_i})\la\zh_{k,1},\xi_{j}^{X_{\hat\imath}}\ra\rp\;.\ee It now suffices to observe that the sum (\ref{LadMorC23}) and the sum (\ref{LadMorC25}) or (\ref{LadMorC26}) cancel out. Indeed, if the morphism condition (\ref{LadMorCond2}) is satisfied, the difference (\ref{LadMorC20}) vanishes and $Q_E\circ \Phi-\Phi\circ Q_F$ vanishes on all generators. Conversely, if the difference (\ref{LadMorC20}) vanishes, the `sum' of the evaluations (\ref{LadMorC21}), (\ref{LadMorC22}), and (\ref{LadMorC24}) vanishes at any point $x\in M$. As the second factor of each one of these evaluations is an element of $F_{-k,\zf_0(x)}$ and the first an arbitrary element $\zh_{k,1;\zf_0(x)}\in F^*_{-k,\zf_0(x)}$, this means that the condition (\ref{LadMorCond2}) is verified. \end{proof}

\noindent{\bf Addendum.} It is straightforwardly checked that the signs $(\pm_1) - (\pm_3)$ are given by $$(\pm_1)=(-1)^k\ze(\zs), (\pm_2)=(-1)^{\tilde X_i(\tilde X_1+\ldots +\tilde X_{i-1}+k)+1},\;\text{ and }\;(\pm_3)=(-1)^k\ze(\zs)\;,$$ which completes the explanation of Definition \ref{LadMorDef}.

\vskip1cm

\noindent Giuseppe BONAVOLONTA \\University of
Luxembourg\\ Campus Kirchberg, Mathematics Research Unit\\ 6, rue R. Coudenhove-Kalergi, L-1359 Luxembourg
City, Grand-Duchy of Luxembourg
\\Email: giuseppe.bonavolonta@uni.lu \\

\noindent Norbert PONCIN\\University of Luxembourg\\
Campus Kirchberg, Mathematics Research Unit\\ 6, rue R. Coudenhove-Kalergi, L-1359 Luxembourg
City, Grand-Duchy of Luxembourg\\Email: norbert.poncin@uni.lu

\end{document}